\documentclass[12pt]{amsart}    

\usepackage{amsmath,amssymb,latexsym} 
\usepackage{graphicx}
\usepackage{fullpage}
\usepackage[square, numbers]{natbib}   
\usepackage{comment}
\usepackage{color}
\usepackage{cases}
\usepackage[squaren,Gray]{SIunits}

\newtheorem{theorem}{Theorem} 
\newtheorem{corollary}{Corollary}
\newtheorem{proposition}{Proposition}
\newtheorem{claim}{Claim}
\newtheorem{lemma}{Lemma}
\theoremstyle{definition}

\theoremstyle{remark}

\def\setR{\mathcal{R}}
\def\setQ{\mathcal{Q}}

\def\G{\mathcal{G}}
\def\V{\mathcal{V}}
\def\A{\mathcal{A}}

\def\R{\mathbb{R}}

\def\Z{\mathbb{Z}}
\def\T{\mathsf{T}}

\def\dd{\boldsymbol{d}}

\def\qq{\boldsymbol{q}}
\def\rr{\boldsymbol{r}}

\def\xx{\boldsymbol{x}}

\def\yy{\boldsymbol{y}}

\def\dot{\makebox[1ex]\cdot}

\def\zero{\boldsymbol{0}}

\def\ds{\displaystyle}

\def\Min{\operatorname{Min}}

\def\ave{\operatorname{ave}}
\def\Pmaxnoret{\textup{P$^{\max}_{\mbox{\textup{\tiny no return}}}$}}
\def\Pmaxret{\textup{P$^{\max}_{\mbox{\textup{\tiny return}}}$}}
\def\Pavenoret{\textup{P$^{\ave}_{\mbox{\textup{\tiny no return}}}$}}
\def\Paveret{\textup{P$^{\ave}_{\mbox{\textup{\tiny return}}}$}}

\def\eps{\varepsilon}

\def\D{\tau}
\def\gmaxret{g^{\max}_{\mbox{\textup{\tiny return}}}}

\def\gmax{g^{\max}}
\def\gave{g^{\ave}}

\title{Minimizing the waiting time for a one-way  \\ shuttle service}

\author{Laurent Daudet}
\email{laurent.daudet@enpc.fr}

\author{Fr\'ed\'eric Meunier}
\email{frederic.meunier@enpc.fr}

\address{\'{E}cole Nationale des Ponts et Chauss\'ees, CERMICS, 77455 Marne-la-Vall\'ee CEDEX, France}

\begin{document}

\maketitle

\begin{abstract}
Consider a terminal in which users arrive continuously over a finite period of time at a variable rate known in advance. A fleet of shuttles has to carry the users over a fixed trip. What is the shuttle schedule that minimizes their waiting time? This is the question addressed in the present paper. We propose efficient algorithms for several variations of this question with proven performance guarantees. The techniques used are of various types (convex optimization, shortest paths,...). The paper ends with numerical experiments showing that most of our algorithms behave also well in practice.
\end{abstract}

\section{Introduction}

The original motivation of this paper comes from a partnership of the authors with Eurotunnel, the company operating the tunnel under the Channel. Eurotunnel is currently facing an increasing congestion due to the trucks waiting in the terminal before being loaded in the shuttles. A way to address this issue consists in scheduling the shuttles so that the trucks do not wait too long in the terminal.

In railway transportation, a traditional point of view considers that the demand can be smoothed by offering sufficiently enough departures over a day. Timetabling is then guided by other considerations, such as robustness, maintainability, or rolling stock. For instance, Swiss, Dutch and German companies usually design periodic timetables, which present many advantages~\cite{cordone2011optimizing, kroon2009new}. The way to optimize this kind of timetables has been the topic of many researches, initiated by Serafini and Ukovich \cite{serafini1989mathematical} and by Voorhoeve \cite{voorhoeve1993rail} explicitly in the railway context, see \cite{kroon2003variable, liebchen2003finding, liebchen2002case, nachtigall1996genetic} for further works. In the context of periodic timetables, a way to adapt the schedules to a demand with strong variations consists in inserting new departures at peak-hours and deleting departures when the demand is low.

Since the trip of the trucks in the tunnel is a small part of their whole journey, it is a reasonable approximation to assume that they cannot choose their arrival time in the terminal. Moreover, increasing the size of the fleet is not always doable in practice (the shuttles are expensive and the tunnel is used by other vehicles, which limits the maximal number of shuttle trips over a day). We face thus a different problem than the one addressed in the aforementioned literature: the demand is assumed to be fixed and nonelastic to the departures times, and the number of shuttles cannot be adjusted to the demand. Given a fleet of shuttles and a demand of transportation known in advance, the problem consists in designing a schedule for the shuttles that minimizes the waiting time of the users. There are timetabling problems with similar features, see  \cite{barrena2014exact, cacchiani2008column, cacchiani2010non, caprara2002modeling, cai1998greedy, ingolotti2006new} for instance, but these articles are at a more macroscopic level than what we require to solve our problem. Moreover, in the present work, the schedules have to be designed in an offline manner. In a transportation context, and especially for Eurotunnel, computing the schedule in advance is mandatory. 

We study several versions of the problem, mainly according to two features. The first feature is whether the shuttles are allowed to come back at the terminal after having realized the trip. The second feature is the objective function. We consider in turn the following two quantities to be minimized: the maximum waiting time and the average waiting time. The first objective is perhaps a fairer one regarding the users, while the second one is relevant for the global efficiency.

It seems that the question we address in the present paper is new. Moreover, it may be relevant for any situation where a demand, known in advance, has to be processed by batches and for which we want to minimize the processing time. This kind of situation is often met in chemical industry. An example whose motivation is very close to ours considers a test to be performed on samples, which arrive continuously~\cite{brauner2007scheduling}. The test takes a certain amount of time and can be performed on several samples simultaneously. The question is then to determine the test schedule that minimizes the processing time.

We propose efficient algorithms for the different versions. ``Efficient'' here means ``theoretically efficient'', according to their time complexity and the performance guarantee. It also means ``practically efficient'' for almost all cases, as shown by numerical experiments conducted on real-world instances. It might also be worth noting that, depending on the version considered, the proof techniques rely on various fields of optimization (convexity, Karush-Kuhn-Tucker conditions, binary search, domination results, shortest paths in finite graphs,...).

\section{Model}

\subsection{The problems}
We are given a fleet of $S$ shuttles, for which departure times have to be determined. All shuttles have a capacity $C\geq 0$ and are situated in the same loading terminal at the beginning of the day. The users are infinitesimal and arrive continuously in the terminal over a finite period of time, modeled as the interval $[0,T]$, following the cumulative nondecreasing function $D:[0,T]\rightarrow\R_+$, where $D(t)$ is the total number of users arrived in the terminal during the interval $[0,t]$. We assume throughout the paper that $D(T)>0$. The shuttles have to carry the users over a fixed trip.

When the users arrive in the terminal, they enter a queue. This queue closes when all the users who will leave with the next shuttle have arrived in the queue and users can enter a new queue only if the previous one is closed (this is how it works at Eurotunnel). When a queue is closed, the users in that queue can start boarding the shuttle. The process in illustrated on Figure~\ref{fig:loading}. Loading a shuttle with a total of $x$ users takes a time $\nu x$. Note that setting $\nu$ to zero allows to model the case where the users do not have to wait for the last user before boarding. Even if no users arrive strictly after time $T$, loading and departures are allowed after that instant. 

\begin{figure}
\begin{center}
\includegraphics[width=12cm]{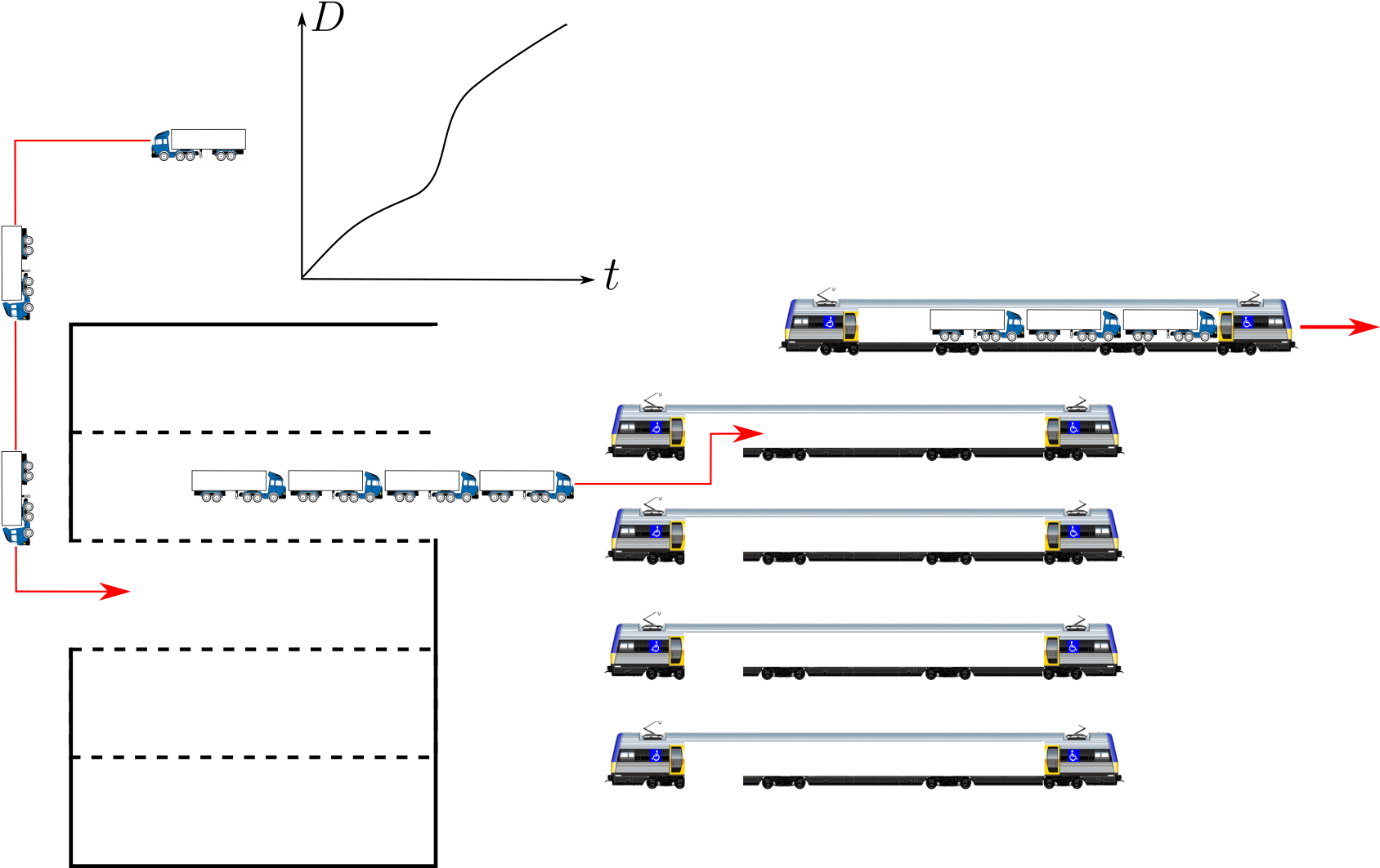}
\caption{The process of arrival, loading, and departure in the terminal.}\label{fig:loading}
\end{center}
\end{figure}

Two possibilities are considered regarding the shuttles. Either return is not allowed: once the shuttles leave, they never come back to the terminal; or it is allowed: once the shuttles leave, they come back to the terminal after a time equal to $\pi\geq 0$. Two objective functions to be minimized are considered: the maximum waiting time and the average waiting time.

We have thus four problems:
\begin{itemize}
\item \Pmaxnoret, which consists of not allowing return and minimizing the maximum waiting time.
\item \Pavenoret, which consists of not allowing return and minimizing the average waiting time.
\item \Pmaxret, which consists of allowing return and minimizing the maximum waiting time.
\item \Paveret, which consists of allowing return and minimizing the average waiting time.
\end{itemize} 

Practical constraints impose that overtake is not possible and thus, when return is allowed, the departure orders of the shuttles remain the same over the whole period. It is nevertheless possible to have simultaneous trips. This is an approximation in the case of Eurotunnel (we neglect the security distance and the length of the shuttles). For other situations, as the chemical application mentioned in the introduction, it may match what is met in practice.

\subsection{The demand}

Throughout the paper, we assume that $D(\dot)$ is upper semicontinuous. It allows to model discontinuity in the arrival process (batch of users arriving simultaneously). Yet, a weaker requirement could lead to mathematical difficulties, e.g., nonexistence of optimal solutions even for very simple cases.

The {\em pseudo-inverses} of $D(\dot)$, defined by 
$$\D\colon y\in[0,D(T)]\longmapsto \left\{\begin{array}{ll}\inf\left\{t\in[0,T]\colon D(t)>y\right\}\in[0,T]&\mbox{if}\;y<D(T)\\T&\mbox{otherwise}\end{array}\right.$$
and
$$\bar\D\colon y\in[0,D(T)] \longmapsto  \inf\left\{t\in[0,T]\colon D(t)\geq y\right\}\in[0,T],$$
play an important role in the paper. Note that they are nondecreasing functions and that $\D(y)\geq\bar\D(y)$ for all $y\in[0,D(T)]$. Times $\D(y)$ and $\bar\D(y)$ are respectively interpreted as the arrival times of the first user after the $y$ firsts and of the last user of these $y$ firsts. Since $D(\dot)$ is upper semicontinuous, we have the following properties, with proofs for sake of completeness.

\begin{lemma}\label{lem:pseudo}
We have $D(\D(y))\geq D(\bar\D(y))\geq y$ for every $y\in[0,D(T)]$.
\end{lemma}

\begin{proof}
Since $D(\dot)$ is nondecreasing, the first inequality is a direct consequence of the inequality $\D(y)\geq\bar\D(y)$, which is obvious from the definition. To prove the second inequality, consider $(t_n)$ a nonincreasing sequence converging toward $\bar\D(y)$ such that $D(t_n)\geq y$ for all $n$. By the upper semicontinuity of $D(\dot)$, we get then $D(\bar\D(y))\geq y$.
\end{proof}

\begin{lemma}\label{lem:semicont}
$\D(\dot)$ is upper semicontinuous and $\bar\D(\dot)$ is lower semicontinuous.
\end{lemma}

\begin{proof}
We first prove that $\D(\dot)$ is upper semicontinuous. Let $\alpha$ be some real number such that $\{y\colon\D(y)<\alpha\}$ is nonempty and take from it an arbitrary element $y_0$. We want to prove that $\{y\colon\D(y)<\alpha\}$ is open for the induced topology on $[0,D(T)]$. If $y_0=D(T)$, then this set is $[0,D(T)]$ and thus open. Otherwise, by the definition of $\D(\dot)$, we know that there exists $t_0<\alpha$ such that $D(t_0)>y_0$. For any element $y$ in $[0,D(t_0))$, we have $\D(y)\leq t_0$, and thus $[0,D(t_0))$ is an open set containing $y_0$ and fully contained in $\{y\colon\D(y)<\alpha\}$. The set $\{y\colon\D(y)<\alpha\}$ is thus an open set of $[0,D(T)]$ for every real number $\alpha$, which precisely means that $\D(\dot)$ is upper semicontinuous.

We prove now that $\bar\D(\dot)$ is lower semicontinuous. Let $\alpha$ be some real number. Consider a converging sequence $(y_n)$ such that $\bar\D(y_n)\leq\alpha$ for all $n$. For every fixed $n$, there exists thus a sequence $(t_{n,k})$ indexed by $k$ such that $D(t_{n,k})\geq y_n$ and $t_{n,k}\leq\alpha+\frac 1 k$ for all $k$. Now, consider the sequence $(t_{n,n})$. It is such that $D(t_{n,n})\geq y_n$ and $t_{n,n}\leq\alpha+\frac 1 n$ for all $n$. Since $[0,T]$ is compact, we can extract a nonincreasing converging subsequence $(t_n)$ from the sequence $(t_{n,n})$ such that $D(t_n)$ converges towards some real number nonsmaller than $\lim_{n\to\infty}y_n$ with $t_n\leq\alpha+\frac 1 n$ for all $n$. It implies that $\bar\D(\lim_{n\to\infty}y_n)\leq \alpha$, which means that $\bar\D(\dot)$ is lower semicontinuous.
\end{proof}

\begin{lemma}\label{lem:increas_semicont}
If $D(\dot)$ is increasing, then $\D(y)=\bar\D(y)$ for every $y\in[0,D(T)]$.
\end{lemma}

\begin{proof}
If $\bar\D(y)=T$, then the equality is obvious. We can thus assume that $\bar\D(y)<T$.
For every $t>\bar\D(y)$, we have $D(t)>D(\bar\D(y))$ since $D(\dot)$ is increasing, and Lemma~\ref{lem:pseudo} implies that $D(t)>y$. By definition of $\D(\dot)$, we have $\D(y)\leq \bar\D(y)$. The reverse inequality being clear from the definitions, we get the result.
\end{proof}

\subsection{Mathematical model}

For the four problems \Pmaxnoret{}, \Pavenoret{}, \Pmaxret{}, and \Paveret{}, a feasible solution is characterized by two nondecreasing sequences of nonnegative real numbers $\dd=d_1,d_2,\ldots$ and $\yy=y_1,y_2,\ldots$. The $d_j$'s are the successive departure times of the shuttles, and the $y_j$'s are their successive cumulative loads: the $j$th departure occurs at time $d_j$ with a load of $y_j-y_{j-1}$ users, where we set $y_0=0$.

Denote by $\gmax(\dd,\yy)$ the value of the maximum waiting time and by $\gave(\dd,\yy)$ the value of the average waiting time. There are explicit expressions of these objective functions. Note that $\D(y_{j})$ can be interpreted as the first arrival time of a user leaving with the ``$(j+1)$th shuttle''. 

\begin{eqnarray*}
\gmax(\dd,\yy)&=&\max_{j\colon y_j>y_{j-1}}\big(d_j-\D(y_{j-1})\big),\\
\gave(\dd,\yy)&=&\frac 1 {D(T)} \sum_j\int_{y_{j-1}}^{y_j}(d_j-\bar\D(y))dy,
\end{eqnarray*}
where the indices $j$ range over all departures.

Problems \Pmaxnoret{} and \Pavenoret{} can be written under the following form,
\begin{equation}\label{Pnoreturn}\tag{P$_{\mbox{\textup{\tiny no return}}}$}
\begin{array}{rl@{\hspace{1cm}}rr}
\Min & g(\dd,\yy) & \\
\mbox{s.t.} & y_j-y_{j-1}\leq C & j=1,\ldots,S & \textup{(i)}\\
& y_{j-1}\leq y_j & j=1,\ldots,S &  \textup{(ii)} \\
& d_{j-1}\leq d_j &   j=2,\ldots,S& \textup{(iii)}\\
& y_S=D(T) &  & \textup{(iv)}\\
& \bar\D(y_j)+\nu(y_j-y_{j-1})\leq d_j & j=1,\ldots,S & \textup{(v)}\\
& y_0=0, &  & 
\end{array}
\end{equation}
where $g(\dot)$ is either $\gmax(\dot)$ or $\gave(\dot)$. Constraint (i) ensures that the total amount of users in any shuttle does not exceed the shuttle capacity. Constraint (ii) ensures that the indices of the $y_j$ variables are consistent. Constraint (iii) ensures that the shuttles do not overtake. Constraint (iv) ensures that every user eventually leaves the terminal in a shuttle. Constraint (v) ensures that the departure time of a shuttle occurs once the last user of this shuttle has arrived and the loading is over.

Problems \Pmaxnoret{} and \Pavenoret{} always admit optimal solutions when they are feasible, i.e., when $CS\geq D(T)$. Indeed, $\bar\D(y_j)+\nu(y_j-y_{j-1})$ is upper-bounded by $T+\nu C$ and adding a constraint $d_j\leq T+\nu C$ for all $j$ does not change the optimal value; since $\bar\D(\dot)$ is lower semicontinuous (Lemma~\ref{lem:semicont}), the set of feasible solutions of the optimization problem obtained with this new contraint is compact; its objective function is lower semicontinuous (and even continuous in the case of \Pavenoret{}).


The following properties for \Pmaxnoret{} and \Pavenoret{} will be useful in some proofs.

\begin{claim}\label{claim:change_obj}
Replacing $\gmax(\dot)$ by $\max_{j}\big(d_j-\D(y_{j-1})\big)$ does not change the optimal value of \Pmaxnoret{}.
\end{claim}

\begin{proof}
Let $(\dd,\yy)$ be a feasible solution of \Pmaxnoret{}. We are going to build a feasible solution $(\dd',\yy)$ (with the same $\yy$) such that 
\begin{equation}\label{eq:comp}
\gmax(\dd,\yy)\geq\gmax(\dd',\yy)=\max_{j}\big(d'_j-\D(y_{j-1})\big).
\end{equation}

We set $d_1'=\bar\D(y_1)+\nu y_1$ and define inductively $d'_j=\max(d_{j-1}',\bar\D(y_j)+\nu(y_j-y_{j-1}))$. We have $d_j'\leq d_j$ for all $j$ and it implies the inequality in \eqref{eq:comp}. Let us prove the equality of \eqref{eq:comp}:
if $\max_{j}\big(d'_j-\D(y_{j-1})\big)$ is attained for a $\bar\jmath$ such that $y_{\bar\jmath-1}<D(T)$, then there exists $k\geq\bar\jmath$ such that $y_k>y_{k-1}=y_{\bar\jmath-1}$ and $d_k'\geq d_{\bar\jmath}'$, which means that the maximum is also attained for a $k$ such that $y_k>y_{k-1}$; and if $\max_{j}\big(d'_j-\D(y_{j-1})\big)$ is attained for a $\bar\jmath$ such that $y_{\bar\jmath-1}=D(T)$, then there exists $\ell\leq \bar\jmath-1$ such that $y_{\ell-1}<y_{\ell}=y_{\bar\jmath-1}$ and by construction $d'_{\ell}=d'_{\ell+1}=\cdots=d'_S$ (since $y_{\ell}=y_{\ell+1}=\cdots=y_S=D(T)$), which means that the maximum is also attained for an $\ell$ such that $y_{\ell}>y_{\ell-1}$.
\end{proof}

\begin{claim}\label{claim:remove_dj}
If $D(\dot)$ is increasing, for any of \Pmaxnoret{} and \Pavenoret{}, there is an optimal solution such that $d_j=\bar\D(y_j)+\nu(y_j-y_{j-1})$ for all $j\in\{1,\ldots,S\}.$
\end{claim}

\begin{proof}
Let $(\dd,\yy)$ be an optimal solution (we do not care which objective function is used yet). Without loss of generality, we can assume that  $d_1=\bar\D(y_1)+\nu y_1$ and that for all $j\in\{2,\ldots,S\}$ we have \begin{equation}\label{eq:dd}d_j=\max\big(d_{j-1},\bar\D(y_j)+\nu(y_j-y_{j-1})\big)\end{equation} (just redefine $d_j$ according to these equalities if necessary). When $\nu=0$, a straightforward induction on $j$ shows that we have then always $d_j=\bar\D(y_j)$. We can thus assume that $\nu>0$.

Suppose for a contradiction that there is a $j$ such that $d_j>\bar\D(y_j)+\nu(y_j-y_{j-1})$. Denote by $j_1$ the smallest index for which this inequality holds. We necessarily have $d_{j_1}=d_{j_1-1}$ (because of the equality~\eqref{eq:dd}). Denote by $j_0$ the smallest index $j< j_1$ such that $d_j=d_{j_1}$. Note that since $D(\dot)$ is increasing, we have that $\bar\D(\dot)$ is continuous (it is upper and lower semi-continuous with Lemma~\ref{lem:increas_semicont}).

For some small $\eps>0$, we define $(\bar\dd,\bar\yy)$ as follows:
$$\bar y_j=\left\{\begin{array}{ll} 
y_j-\eps & \mbox{for $j=j_0,\ldots,j_1-1$} \\
y_j & \mbox{otherwise}
\end{array}\right.$$
and
$$\bar d_j=\left\{\begin{array}{ll}  
\max\big(\bar d_{j-1},\bar\D(\bar y_j)+\nu(\bar y_j-\bar y_{j-1})\big) &  \mbox{for $j=j_0,\ldots,j_1$}\\
d_j &  \mbox{otherwise,}
\end{array}\right.$$
where $\bar d_0=0$. We first  check that $(\bar\dd,\bar\yy)$ is a feasible solution of \eqref{Pnoreturn}.

The definition of $j_1$ implies that $d_{j_0}>0$. Thus if $j_0=1$, we have $y_1>0$ and for a small enough $\eps$, the vector $\bar\yy$ satisfies constraint (ii). Otherwise, we have $\bar\D(y_{j_0-1})+\nu(y_{j_0-1}-y_{j_0-2})=d_{j_0-1}<d_{j_0}=\bar\D(y_{j_0})+\nu(y_{j_0}-y_{j_0-1})$. It implies that $y_{j_0-1}<y_{j_0}$ (as otherwise the equality would imply that $y_{j_0-1}<y_{j_0-2}$). Thus, for a small enough $\eps$, we have $\bar\yy$ satisfies constraint (ii). It also satisfies obviously constraint (iv).

For $j\in\{2,\ldots,j_1\}\cup\{j_1+2,\ldots,S\}$, checking $\bar d_{j-1}\leq \bar d_j$ is straightforward. The remaining case is $j=j_1+1$. A direct induction shows that $\bar d_j\leq d_j$ for $j\leq j_1-1$. Since $\bar\D(y_{j_1})+\nu(y_{j_1}-y_{j_1-1})<\bar\D(y_{j_1-1})+\nu(y_{j_1-1}-y_{j_1-2})$ (because $d_{j_1-1}=d_{j_1}$), for $\eps$ small enough, we have $\bar d_{j_1-1}\geq \bar\D(\bar y_{j_1})+\nu(\bar y_{j_1}-\bar y_{j_1-1})$. Here, we use the fact that $\bar\D(\dot)$ is continuous. Thus $\bar d_{j_1}=\bar d_{j_1-1}$. Since we have $\bar d_{j_1-1}\leq d_{j_1-1}$ by the above induction, we finally obtain $\bar d_{j_1}\leq d_{j_1}\leq d_{j_1+1}=\bar d_{j_1+1}$. Therefore, $\bar\dd$ satisfies constraint (iii).

Constraint (i) is satisfied for all $j$, except maybe for $j=j_1$. We have proved that $\bar d_{j_1}=\bar d_{j_1-1}$. Since $\bar d_{j_1-1}=\bar\D(\bar y_j)+\nu(\bar y_j-\bar y_{j-1})$ for some $j'\leq j_1-1$, we have $\bar\D(\bar y_{j_1})+\nu(\bar y_{j_1}-\bar y_{j_1-1})\leq\bar d_{j_1}=\bar\D(\bar y_{j'})+\nu(\bar y_{j'}-\bar y_{j'-1})$, and thus $\nu(\bar y_{j_1}-\bar y_{j_1-1})\leq\nu(\bar y_{j'}-\bar y_{j'-1})\leq \nu C$. Therefore constraint (i) is also satisfied for $j=j_1$.

Since the constraint (v) is clearly satisfied, $(\bar\dd,\bar\yy)$ is a feasible solution of \eqref{Pnoreturn}.

A careful examination of the arguments used when we checked constraint (ii) shows that actually $\bar d_{j_0}< d_{j_0}$. The same induction as the one used we checked constraint (iii) shows that $\bar d_{j_1-1}<d_{j_1-1}$. We have proved that $\bar d_{j_1-1}\geq\bar\D(\bar y_{j_1})+\nu(\bar y_{j_1}-\bar y_{j_1-1})$. Thus $\bar d_{j_1}=\bar d_{j_1-1}$, and $\bar d_{j_1}<d_{j_1}$. We have
$$
\sum_{j=j_0}^{j_1}\int_{\bar y_{j-1}}^{\bar y_j}\big(\bar d_j-\bar\D(u)\big)du\leq\int_{\bar y_{j_0-1}}^{\bar y_{j_1}}\big(\bar d_{j_1}-\bar\D(u)\big)du<\int_{y_{j_0-1}}^{y_{j_1}}\big(d_{j_1}-\bar\D(u)\big)du=\sum_{j=j_0}^{j_1}\int_{y_{j-1}}^{y_j}\big(d_j-\bar\D(u)\big)du,$$ which in contradiction with the optimality assumption. This settles the case of $\gave(\cdot)$. The other case is dealt with similarly.
\end{proof}

Problems \Pmaxret{} and \Paveret{} can be written almost identically under the following form. We use infinitely many variables since there is no \textit{a priori} reason to have a bounded number of departures, and there are indeed special cases for which there is no optimal solution with a finite number of departures. However, if $\pi>0$, we prove that any optimal solution of \Pmaxret{} requires a finite number of departures, see Proposition~\ref{prop:finite_max}. The case of \Paveret{} remains open.
\begin{equation}\label{Preturn}\tag{P$_{\mbox{\textup{\tiny return}}}$}
\begin{array}{rl@{\hspace{1cm}}rr}
\Min & g(\dd,\yy) & \\
\mbox{s.t.} & y_j-y_{j-1}\leq C & j=1,\ldots,+\infty & \textup{(i)}\\
& y_{j-1}\leq y_j & j=1,\ldots,+\infty &  \textup{(ii)} \\
& d_{j-1}\leq d_j &   j=2,\ldots,S & \textup{(iii)}\\
& \ds{\lim_{j\rightarrow+\infty}y_j=D(T)} &  & \textup{(iv)}\\
& \bar\D(y_j)+\nu(y_j-y_{j-1})\leq d_j & j=1,\ldots,+\infty & \textup{(v)}\\
& d_j+\pi+\nu(y_{j+S}-y_{j+S-1})\leq d_{j+S}& j=1,\ldots,+\infty & \textup{(vi)} \\
& y_0=0, &  & 
\end{array}
\end{equation}
where $g(\dot)$ is either $\gmax(\dot)$ or $\gave(\dot)$. Constraints (i), (ii), (iii), (iv), and (v) have the same meaning as for the previous problems. Constraint (vi) ensures that the time between two consecutive departures of a same shuttle is not smaller than the time required for a full trip plus the time needed to load the users.

In the model \eqref{Preturn}, the shuttles are not identified. Note however that their schedules can be easily be recovered: the departure times of a shuttle $s$ is of the form 
$$d_s,d_{s+S},d_{s+2S},\ldots$$ and the time at which the loading starts for a shuttle with departure time $d_j$ can be chosen to be $d_j-\nu(y_j-y_{j-1})$ (the loading starts as late as possible).

While it can be shown that problem \Pmaxret{} always admits an optimal solution when it is feasible (see Proposition~\ref{prop:finite_max}), we were not able to settle the case of problem \Paveret{}.


\subsection{Computational model}

We assume that the following operations take constant time:
\begin{itemize}
\item Evaluation of $D(t)$ for any $t\in[0,T]$.
\item Integration of $D(\dot)$ between two values.
\item Evaluation of  $\D(y)$ and $\bar\D(y)$ for any $y\in\R_+$.
\item Evaluation of $\sup\{y\colon\bar\D(y)+\nu y\leq \alpha\}$ for any $\alpha\in\R_+$.
\end{itemize} Note that if $D(\dot)$ is piecewise affine with a natural description, as it is usually the case in practice, these assumptions are easily matched. Moreover, we set as constants of the computational model the capacity $C$, the length of the period $T$, the cumulative demand $D(\dot)$, the loading rate $\nu$, and the return time $\pi$. The complexity functions will be expressed in terms of $S$ and the accuracy of the computed solution.

\section{Main results}\label{sec:mainresults}

In the present section, we present our main findings. Many results state the existence of algorithms with a guarantee that the returned solution has a value close to the optimal value $OPT$ of the considered problem. Except for two easy results -- Corollary~\ref{cor:approx} and Proposition~\ref{prop:finite_max} -- all proofs are postponed to other sections.

We organize the results presented in that section in three subsections. The first subsection -- Section~\ref{subsec:one} -- deals with the special case where $D(\dot)$ is a constant function, i.e., when all users are in the loading terminal from the beginning of the period, and with returns allowed. 
It seems to us that these results are also interesting in themselves, because they form a natural situation for which there is a very efficient algorithm. The second subsection -- Section~\ref{subsec:noret} -- deals with the general case where the shuttles are not allowed to come back, i.e., with the case covered by the problems \Pmaxnoret{} and \Pavenoret{}. The case where the shuttles are allowed to come back, i.e., when we deal with the problems \Pmaxret{} and \Paveret{}, is discussed in Section~\ref{subsec:ret}. 

\subsection{All users in the terminal from the beginning}\label{subsec:one}

In this subsection, we present results regarding the four problems when $D(t)=D(T)$ for all $t\in[0,T]$ (all users are from the beginning in the terminal). For the problems for which return is not allowed (\Pmaxnoret{} and \Pavenoret), an obvious optimal solution is given by $y_j^*=jD(T)/S$ and $d_j^*=\nu D(T)/S$ for $j\in\left\{1,\ldots,S\right\}$ and the optimal value is $\nu D(T)/S$ for both problems, provided that $D(T)\leq CS$ (otherwise, there is no feasible solution at all): the shuttles take all the same amount of users, start immediately the loading process, and have the same departure time.

The rest of the section is devoted to the results regarding the problems \Pmaxret{} and \Paveret. For the first one, there are closed-from expressions for the optimal value and an optimal solution.

\begin{proposition}\label{prop:S1max}
When $D(t)=D(T)$ for all $t\in[0,T]$, the optimal value of \Pmaxret{} is $$\frac{\nu D(T)}{S}+\left(\left\lceil \frac{D(T)}{CS}\right\rceil-1\right)\pi.$$ 
\end{proposition}

In the proof, we actually provide a closed-form expression for an optimal solution. For \Paveret{} however, there does not seem to be a closed-form expression for an optimal solution, and not even for the optimal value. There is nevertheless an efficient algorithm.

\begin{proposition}\label{prop:S1}
Suppose $\pi>0$. When $D(t)=D(T)$ for all $t\in[0,T]$, the optimal value of \Paveret{} can be computed in constant time and an optimal solution can be computed in $O(S)$. 
\end{proposition}

 If $\pi=0$, the optimal value is $\frac {\nu D(T)}{2S}$, and it is not too difficult to see that there is no optimal solution. In a transportation context, $\pi=0$ looks unrealistic. However, the chemical application mentioned in the introduction could be a situation where this equality could be met: as soon as the test is over for a batch, we can start the test for a new one.

\subsection{When return is not allowed}\label{subsec:noret}
We have the existence of an efficient approximation algorithm for \Pmaxnoret{}. The algorithm is actually an easy binary search (Section~\ref{subsubsec:algo}).

 \begin{theorem}\label{thm:pmax}
 Let $\rho>0$. A feasible solution $(\dd,\yy)$ of \Pmaxnoret{} -- if the problem is feasible -- satisfying $\gmax(\dd,\yy)\leq OPT + \rho$ can be computed in $O\left(S\log{\frac 1 \rho}\right)$.
 \end{theorem} 
 With an additional assumption on $D(\dot)$, this theorem provides actually an approximation scheme.
 
 \begin{corollary}\label{cor:approx}
If $D(\dot)$ is increasing, the algorithm of Theorem~\ref{thm:pmax} computes in $O(S\log\frac S {\eps})$ a $(1+\varepsilon)$-approximation for \Pmaxnoret.
\end{corollary}
A schedule for the shuttles requires to specify $S$ real numbers. Taking an output sensitive point of view, this corollary states thus the existence of a polynomial approximation scheme in this particular case.

\begin{proof}[Proof of Corollary~\ref{cor:approx}]
Suppose $D(\dot)$ increasing. Let $(\dd,\yy)$ be a feasible solution. According to Lemma~\ref{lem:increas_semicont}, we have then $\D(y_{j-1})=\bar\D(y_{j-1})$ for every $j$ and the maximum waiting time for shuttle $j$ is at least $\D(y_j)+\nu(y_j-y_{j-1})-\D(y_{j-1})$. Note that if $y_j=y_{j-1}$, this quantity is zero. Hence, the sum of the maximum waiting times over all nonempty shuttles is at least $T+\nu D(T)$ and the optimal value $OPT$ of \Pmaxnoret{} is at least $(T+\nu D(T))/S$. Setting $\rho$ to $\eps(T+\nu D(T))/S$ leads to the result.
\end{proof}

For \Pavenoret{}, there exists an efficient approximation algorithm too. The algorithm is also described later (Section~\ref{subsubsec:algo_ave}). We already outline that this algorithm is not a binary search as in the former case, but consists in building a quite simple weighted ``approximative'' graph, in which a shortest path is computed.

\begin{theorem}\label{thm:pave}
Suppose that $D(\dot)$ admits right derivatives everywhere (denoted $D'_+(t)$) and that $\inf_{t\in[0,T)}D'_+(t)$ is positive. Then, for any positive integer $M$, a feasible solution $(\dd,\yy)$ of \Pavenoret{} -- if the problem is feasible -- satisfying $$\gave(\dd,\yy)\leq OPT+O\left(\frac {S^2} M\right)$$ can be computed in $O\left(SM^3\right)$.
\end{theorem}

As for Corollary~\ref{cor:approx} above, this theorem could be interpreted as a polynomial approximation scheme by using the fact that $D(\dot)$ is increasing.

\subsection{When return is allowed}\label{subsec:ret}

The following proposition implies that when $\pi$ is larger than $0$, any optimal solution of \Pmaxret{} requires a finite number of nonempty departures.

\begin{proposition}\label{prop:finite_max}
If $\pi>0$, there exists an optimal solution of \Pmaxret{} and the number of nonempty departures in any optimal solution is at most
$$\left(2\left\lceil\frac {T} \pi \right\rceil +1\right)S+\left(\frac{\nu}{\pi}+\frac 1 C\right)D(T).$$
\end{proposition}

\begin{proof}
A feasible solution with an infinite number of nonempty departures has an infinite objective value and it thus strictly dominated by any solution by a finite number of departures. Thus, the set of feasible solutions can be reduced to the solutions where the number of nonempty departures is finite. Since $\bar\D(\dot)$ is lower semicontinuous (Lemma~\ref{lem:semicont}), the set of feasible solutions is compact and the objective function is lower semicontinuous which leads then to existence of an optimal solution. 

The schedule consisting in making the shuttles wait until time $T$, loading them at full capacity (except maybe for the last departure), and making them leave as soon as the loading is completed provides a feasible solution of \Pmaxret{} with a value $T+\nu D(T)/S+(\lceil D(T)/(CS)\rceil-1)\pi$.

Consider an optimal solution of \Pmaxret{} and denote by $k$ the number of departures after time $T$. The users in the last shuttle to leave have waited at least $(k/S-1)\pi$. We have thus 
$$\left(\frac k S-1\right)\pi\leq T+\frac{\nu D(T)}S+\frac{\pi D(T)}{CS},$$ which implies that $$k\leq\frac {TS} \pi+\frac{\nu D(T)}{\pi}+\frac{D(T)}{C}+S.$$
Before time $T$, the number of departures is at most $\lceil T/\pi\rceil S$.
\end{proof}

The next theorem states that there exists an algorithm computing arbitrarily good feasible solutions for \Pmaxret{} within reasonable computational times when $S$ is small. As for Section~\ref{subsec:noret}, this algorithm is described later in the paper (Section~\ref{sec:ret}). It is based on the computation of a shortest path in an ``approximative'' graph, as for Theorem~\ref{thm:pave}. It also uses Proposition~\ref{prop:finite_max} in a crucial way (actually a slight variation of it: Lemma~\ref{lem:t+}).

\begin{theorem}\label{thm:pmaxret}
Suppose that $D(\dot)$ admits right derivatives everywhere, $\pi$ is positive, and $\inf_{t\in[0,T)}D'_+(t)$ is positive. Then, for any positive integer $M$, a feasible solution $(\dd,\yy)$ of \Pmaxret{} satisfying $$\gmaxret(\dd,\yy)\leq OPT+O\left(\frac {S^2} M\right)$$ can be computed in $O\left(\beta^{3S}M^{3S+2}\right)$, where $\beta$ depends only on the constants of the computational model.
\end{theorem}

As above, the theorem actually ensures that the algorithm is an approximation scheme since we can bound from below $OPT$ using only the input values. If $S$ is considered as constant, this becomes even a polynomial approximation scheme.\\

We do not know whether there is a counterpart to Proposition~\ref{prop:finite_max} for problem \Paveret{}. If such a counterpart would exist, then almost the same technique as the one used in Section~\ref{sec:ret} would lead to a theorem similar to Theorem~\ref{thm:pmaxret} for \Paveret{}. The existence of such a theorem remains thus open.

\section{All users in the terminal from the beginning}

Consider the case where all users are in the loading terminal from the beginning. To ease the reading, and for the present section only, we use $D$ to denote the quantity $D(T)$. 

We treat first the case of problem \Pmaxret{}.  

\begin{proof}[Proof of Proposition~\ref{prop:S1max}]
For \Pmaxret{}, when $S=1$, an optimal solution is obtained by loading at full capacity the shuttle for each departure (except maybe for the last departure for which the shuttle load is $D-C\lfloor D/C\rfloor$) and by making the shuttle leave immediately after each loading process. The optimal value is then $\nu D+(\lceil D/C\rceil-1)\pi$. When $S>1$, consider the problem Q defined as the problem \Pmaxret{} without the constraint that the shuttles do not overtake (constraint (iii) in \eqref{Preturn}). The optimal value of Q provides a lower bound of the optimal value of \Pmaxret{}. Since there is no constraint linking the different shuttles, problem Q can be solved separately for each shuttle $s$ with a demand $D_s$ to carry, such that $\sum_s D_s=D$. The optimal solutions of Q are thus obtained from the optimal solutions of 
$$\begin{array}{rll}
\Min & \ds{\max_{s\in\left\{1,\ldots,S\right\}}\left(\nu D_s+\left(\left\lceil\frac{D_s}C\right\rceil-1\right)\pi\right)} \\ 
\mbox{s.c.} & \ds{\sum_{s=1}^SD_s=D} & \\
& D_s\geq 0 &  s=1,\ldots,S.
\end{array}$$
The solution given by $D_s=D/S$ for all $s$ is clearly optimal (and it is actually the unique optimal solution when $\nu>0$). Hence, there is an optimal solution for Q in which all shuttles have the same departure times and, for each travel, carry the same amount of users. Its value is $\nu D/S+(\lceil D/(CS)\rceil-1)\pi$. Since the shuttles do not overtake in this optimal solution of Q, it is actually a feasible solution for the problem \Pmaxret{}, and thus an optimal solution for this latter problem (its value being equal to a lower bound).
\end{proof}

The rest of this section is devoted to the proof of Proposition~\ref{prop:S1}, which ensures the existence of an efficient algorithm solving problem \Paveret{} when $D(t)=D$ for all $t\in[0,T]$. We start by considering the special case of problem \Paveret{} when $S=1$. In such a case, it is always beneficial to define $d_j=(j-1)\pi+\nu y_j$. Assuming that the $y_j$'s are given, it provides a feasible solution since $\bar\D(y)=0$ for all $y\in[0,D]$.
The objective function of \Paveret{} becomes thus 
$$\frac 1 D\sum_{j=1}^{+\infty}\left((j-1)\pi+\nu \sum_{i=1}^{j}x_i\right)x_j=\frac 1 D\left(\sum_{j=1}^{+\infty}(j-1)\pi x_j+\frac 1 2\nu\sum_{j=1}^{+\infty}x_j^2\right)+\frac {\nu D} 2 ,$$ where $x_j=y_j-y_{j-1}$. Solving \Paveret{} when $S=1$ reduces thus to solving
 \begin{equation}\label{S1averet}\tag*{$P(D)$}
\begin{array}{rlr}
\Min & \ds{\sum_{j=1}^{+\infty}(j-1)\pi x_j+\frac 1 2\nu\sum_{j=1}^{+\infty}x_j^2} \\
\mbox{s.t.} & \ds{\sum_{j=1}^{+\infty}x_j=D} \\
& 0\leq x_j\leq C & j=1,\ldots,+\infty,
\end{array}
\end{equation}
which is a convex program (with infinitely many variables). We will show that there is always an optimal solution of \ref{S1averet} with a finite support. Then, we will solve $P_0(D)$, defined as the program~\ref{S1averet} with the additional constraint $|\{j\colon x_j\neq 0\}|<+\infty$, with the help of the Karush-Kuhn-Tucker conditions (that do not apply otherwise).

\begin{lemma}\label{lem:y_p0}
Suppose that $\pi>0$. Then $P_0(D)$ has an optimal solution and it is necessarily of the form
$$\begin{array}{l}
x_0^*=0 \\
x_j^*=\left\{
\begin{array}{ll}
C & \quad\mbox{if $j\leq a$,} \smallskip\\ 
\ds{\frac{D-aC}{\theta(a)-a}+\frac\pi\nu\left(\frac{a+\theta(a)+1} 2-j\right)} & \quad\mbox{if $a+1\leq j\leq \theta(a)$,} \smallskip
\\ 
0 & \quad\mbox{otherwise,}
\end{array}\right.
\end{array}$$
with $a\in\Z_+$ such that $a\leq \frac DC$ and where $$\theta(a)=a+\left\lceil\frac{-1+\sqrt{1+\frac{8\nu C}\pi(D-aC)}}{2}\right\rceil.$$
\end{lemma}

\begin{proof}
Consider the following program 
\begin{equation}\label{P0Dn}\tag*{$P_0^n(D)$}\begin{array}{rlr}
\Min & \ds{\sum_{j=1}^{n}(j-1)\pi x_j+\frac 1 2\nu\sum_{j=1}^{n}x_j^2} \\
\mbox{s.t.} & \ds{\sum_{j=1}^nx_j=D} \\
& 0\leq x_j\leq C & j=1,\ldots,n.
\end{array}\end{equation}
Note that \ref{P0Dn} is actually $P_0(D)$ with the additional constraint that $\sup\{j\colon x_j\neq 0\}\leq n$.

For $n<D/C$, \ref{P0Dn} has no feasible solutions, and for $n\geq D/C$, the set of feasible solutions is nonempty. In this case, by compactness and continuity of the objective function, \ref{P0Dn} has an optimal solution $\xx^*$. We necessarily have $x_j^*\geq x_{j+1}^*$ for every $j\in\left\{1,\ldots,n-1\right\}$, otherwise, exchanging the two values would strictly decrease the objective function. Let $a$ be the largest index $j$ such that $x_j^*=C$, with the convention that $a=0$ if there are no such index $j$, and let $b+1$ be the smallest index $j$ such that $x_j^*=0$, with the convention that $b=n$ if there is no such index $j$.

The constraints being all affine, the Karush-Kuhn-Tucker conditions apply. There is thus a real number $\lambda\in\R$ and two collections $\boldsymbol{\mu}, \boldsymbol{\omega}\in\R_+^n$ such that  
for every $j\in\{1,\ldots,n\}$ we have 
\begin{equation}\label{eq:kkt}
\nu x_j^*+(j-1)\pi+\lambda+\mu_j-\omega_j=0 \qquad \mbox{and} \qquad \omega_jx^*_j=\mu_j(x_j^*-C)=0.
\end{equation}
Summing this equality from $j=a+1$ to $j=b$ and noting that $\mu_j=\omega_j=0$ and $\sum_{j=a+1}^bx_j^*=D-aC$ by definition of $a$ and $b$ provide an expression of $\lambda$ in terms of $a$ and $b$. Replacing $\lambda$ by this expression in the same equality leads to
\begin{equation*}
x_j^*=\left\{
\begin{array}{ll}
C & \quad\mbox{if $j\leq a$,} \smallskip\\ 
\ds{\frac{D-aC}{b-a}+\frac\pi\nu\left(\frac{a+b+1} 2-j\right)} & \quad\mbox{if $a+1\leq j\leq b$,} \smallskip
\\ 
0 & \quad\mbox{otherwise.}
\end{array}\right.
\end{equation*}
Using this equality for $j=b$ gives the following equation.
$$(b-a)(b-a-1)<\frac{2\nu}\pi(D-aC).$$
Equation~\eqref{eq:kkt} specialized for $j=b+1$ gives $$(b-a)(b-a+1)\geq\frac{2\nu}\pi(D-aC).$$
These two inequalities together -- treated as conditions on a second order polynomial in $b-a$ -- imply the necessary condition
$$
-\frac{1}2+\frac{\sqrt{1+\frac{8\nu}\pi(D-aC)}}2\leq b-a<\frac 12+\frac{\sqrt{1+\frac{8\nu}\pi(D-aC)}}2
$$
which imposes a unique integer value for $b-a$ and $b$ takes a unique value $\theta(a)$ for each value of $a$.
We have proved that any optimal solution of \ref{P0Dn} is of this form. Now, note that by definition of $a$, we necessarily have $a\leq \lfloor D/C\rfloor$. It means that there are only finitely many optimal solutions of the \ref{P0Dn}'s when $n$ goes to infinity. Since the set of feasible solutions of the \ref{P0Dn}'s is nondecreasing when $n$ goes to infinity, it means actually that there exists an $n_0$ such that any optimal solution of \ref{P0Dn} for $n\geq n_0$ is an optimal solution of $P_0^{n_0}(D)$. Moreover, any feasible solution of $P_0(D)$ is a feasible solution of $P_0^n(D)$ for some $n\geq n_0$, and thus is dominated by the optimal solutions of $P_0^{n_0}(D)$. These latter are thus the optimal solutions of $P_0(D)$.
\end{proof}

Let $v(D)$ and $v_0(D)$ be the optimal values of respectively \ref{S1averet} and $P_0(D)$. Note that $v(D)\leq v_0(D)$.

\begin{lemma}\label{lem:D2}
If $\pi>0$, we have
$$v_0(D-\varepsilon)\leq v(D)$$ for every $\varepsilon\in(0,D]$.
\end{lemma}

\begin{proof}
Let $\eps\in(0,D]$. Consider a feasible solution $\xx$ of $P(D)$. Let $N_\eps\in\Z_+$ be such that $\sum_{j=N_\eps+1}^{+\infty} x_j<\eps$. Define inductively
$$x'_j=\left\{\begin{array}{ll}\min(x_j,D-\eps-\sum_{i=1}^{j-1}x_i') & \mbox{for $j\leq N_\eps$} \\ 0 & \mbox{for $j\geq N_\eps+1$.}\end{array}\right.$$ This $\xx'$ is a feasible solution of $P_0(D-\eps)$. Since $x_j'\leq x_j$ for all $j$, the value given by $\xx'$ to the objective value of $P_0(D-\eps)$ is nonlarger that the value obtained by $\xx$ for $P(D)$. The inequality follows.
\end{proof}

\begin{proof}[Proof of Proposition~\ref{prop:S1}]
Let us deal with the case $S=1$. Using the fact that $P_0(D)$ is a convex program, we easily get that $v_0(\dot)$ is a convex function. It is thus continuous on $(0,+\infty)$, and since $v_0(0)=0$, we have that $v_0(\dot)$ is continuous everywhere on $[0,+\infty)$. Making $\eps$ tend toward $0$ in Lemma~\ref{lem:D2} and the inequality $v(D)\leq v_0(D)$ imply that $v_0(D)=v(D)$. Since any feasible solution of $P_0(D)$ is a feasible solution of $P(D)$ with the same value for the objective function, every optimal solution of $P_0(D)$ is an optimal solution of $P(D)$. An algorithm computing an optimal solution of $P(D)$ can then be derived from Lemma~\ref{lem:y_p0}: we just have to try all the finitely many possible values for $a$. The proof for any value of $S$ will be obtained by showing that an optimal solution in this case consists just in replicating optimal solutions for the one-shuttle case. 

When $S>1$, consider the problem Q defined as the problem \Paveret{} without the constraint that the shuttles do not overtake (constraint (iii) in \eqref{Preturn}). The optimal value of Q provides a lower bound of the optimal value of \Paveret{}. Since there is no constraint linking the different shuttles, problem Q can be solved separately for each shuttle $s$ with a demand $D_s$ to carry, such that $\sum_s D_s=D$. The optimal solutions of Q are thus obtained from the optimal solutions of 
$$\begin{array}{rll}
\Min & \ds{\sum_{s=1}^S\left(v(D_s)+\frac \nu 2 D_s^2\right)} \\ 
\mbox{s.c.} & \ds{\sum_{s=1}^SD_s=D} & \\
& D_s\geq 0 & \forall s=1,\ldots,S.
\end{array}$$
The fact that $P(D)$ is a convex program implies that the map $v(\dot)$ is convex. As a consequence, the solution $D_s=D/S$ for all $s$ is an optimal solution of the previous program. Hence, there is an optimal solution for Q in which all shuttles have the same departure times and, for each travel, carry the same amount of users. Since the shuttles do not overtake in this optimal solution of Q, it is actually a feasible solution for the problem \Paveret{}, and thus an optimal solution for this latter problem (its value being equal to a lower bound).
\end{proof}

 \newpage

\section{When return is not allowed}\label{sec:noret}

\subsection{Minimizing the maximum waiting time}\label{subsec:pmax}

\subsubsection{The algorithm}\label{subsubsec:algo}

If $CS<D(T)$, there is no feasible solution. We can thus assume that $CS\geq D(T)$. The algorithm is a binary search starting with the values $h^+=T+\nu D(T)$ and $h^-=0$ which are respectively upper and lower bounds of the optimal value. While the gap $h^+-h^-$ is larger than $\rho$, we consider the tentative value $h=\frac{h^++h^-} 2$ and the system
\begin{equation}\label{Smaxnoreth}\tag{S$_h$}
\left\{
\begin{array}{l@{\hspace{1cm}}r}
y_j = \sup \mathcal{S}_j^h &  j=1,\ldots,S\\
y_S=D(T) &\\
y_0 = 0\\
d_j = h+\D(y_{j-1}) & j=1,\ldots,S,
 \end{array}\right.
\end{equation}
where $\mathcal{S}_j^h=\left\{y\in\R_+ \colon y\leq C+y_{j-1}, \bar{\D}(y)+\nu(y-y_{j-1})-\D(y_{j-1})\leq h,y\leq D(T)\right\}$.
Each iteration of the binary search consists in deciding whether \eqref{Smaxnoreth} has a feasible solution or not, and it can be done in $O(S)$ by computing the values of the $y_j$'s and the $d_j$'s iteratively (here we use in particular the computational assumptions on $D$). As we are going to prove, \eqref{Smaxnoreth} has a feasible solution if and only if the problem has a feasible solution with a value of the objective function at most $h$. If \eqref{Smaxnoreth} has a feasible solution, we update thus the value of $h^+$ with the current value of $h$, otherwise, we update $h^-$ with $h$. When $h^+-h^-\leq\rho$, the solution of program $(\text{S}_{h^+})$ is feasible for \Pmaxnoret{} and its value $h^+$ is at most at $\rho$ from the optimal value.

\subsubsection{Proof of Theorem~\ref{thm:pmax}}
For any fixed $h$, \Pmaxnoret{} has a feasible solution with a value of the objective function at most $h$ if and only if the following system has a feasible solution.
 \begin{equation}\label{Qmaxnoreth}\tag{Q$_h$}
 \left\{
\begin{array}{l@{\hspace{1cm}}rr}
d_j-\D(y_{j-1})\leq h & j=1,\ldots,S & \textup{(Qi)}\\
 y_j-y_{j-1}\leq C & j=1,\ldots,S & \textup{(Qii)}\\
 y_{j-1}\leq y_j & j=1,\ldots,S & \textup{(Qiii)}\\
 d_{j-1}\leq d_j  &   j=2,\ldots,S& \textup{(Qiv)}\\
 y_S=D(T) &  &\textup{(Qv)}\\
\bar\D(y_j)+\nu(y_j-y_{j-1})\leq d_j& j=1,\ldots,S & \textup{(Qvi)}\\
 y_0=0. & 
\end{array}\right.
\end{equation}
We claim that \eqref{Qmaxnoreth} has a feasible solution if and only if \eqref{Smaxnoreth} has one. Once this equivalence is established, the correctness of the binary search described above is almost immediate using Claim~\ref{claim:change_obj}.

Let $(\dd,\yy)$ be a feasible solution of \eqref{Smaxnoreth}. We use without further mention that $\mathcal{S}_j^h$ is closed. It satisfies the constraints $\textup{(Qi)}, \textup{(Qii)}$, and $\textup{(Qv)}$. We have $y_{j-1}\leq C+y_{j-1}$ and $y_{j-1}\leq D(T)$. Since $\bar{\D}(y)\leq \D(y)$ for all $y$, we also have $\bar{\D}(y_{j-1})+\nu(y_{j-1}-y_{j-1})-\D(y_{j-1})\leq h$. It means that $y_{j-1}$ belongs to $\mathcal{S}_j^h$, and thus $y_{j-1}\leq y_j$. Hence, $(\dd,\yy)$ satisfies also constraint $\textup{(Qiii)}$. Since $\bar{\D}(y_j)+\nu(y_j-y_{j-1})-\D(y_{j-1})\leq h$, the solution also satisfies constraint $\textup{(Qvi)}$ and since $\D(\dot)$ is nondecreasing, it satisfies constraint $\textup{(Qiv)}$. Therefore, it is a feasible solution of \eqref{Qmaxnoreth} and the existence of a feasible solution of \eqref{Smaxnoreth} implies the existence of a feasible solution of \eqref{Qmaxnoreth}.

For the converse implication, suppose that \eqref{Qmaxnoreth} admits a feasible solution, and consider the optimization problem consisting in maximizing $\sum_{j=1}^Sy_j$ over its feasible solutions. These feasible solutions form a compact set of $\R_+^S$ since it is obviously bounded and since the semicontinuities of $\D(\dot)$ and $\bar\D(\dot)$ imply that it is closed. There is thus an optimal solution $(\dd^*,\yy^*)$ to that optimization problem. Suppose for a contradiction that there is a $j$ such that $y_j^*< \sup \mathcal{S}_j^h$. Denote $j_0$ the largest such index. Let us slightly increase $y_{j_0}^*$, while letting the other $y_j^*$ untouched. Redefine $d_j^*$ to be $h+\D(y_{j-1}^*)$ for all $j\geq j_0$. The pair $(\dd^*,\yy^*)$ remains feasible for \eqref{Qmaxnoreth} (we use here the fact that $\sup\mathcal{S}_j^h$ is nondecreasing with $j$), while increasing the quantity $\sum_{j=1}^Sy_j^*$, which is a contradiction with the optimality assumption. Thus, we have $y_j^*=\sup \mathcal{S}_j^h$ for all $j$ and $d_j^*:=h+\D(y_{j-1}^*)$ for all $j$ provides a feasible solution for \eqref{Smaxnoreth}.
\qed

 \subsection{Minimizing the average waiting time}\label{subsec:pave}

 \subsubsection{The algorithm} \label{subsubsec:algo_ave}The following map will be useful in the description of the algorithm.
 $$f^{\ave}:(d,y,y')\longmapsto\int_y^{y'}(d-\bar\D(u))du.$$

 Define the directed graph $\G=(\V,\A)$  by 
 $$\begin{array}{rcl} 
 \V & =  & \{(0,0)\}\cup\{\eta,2\eta,\ldots,M\eta\}\times\{\eta,2\eta,\ldots,R\eta\}\smallskip\\
 \A &  =  & \left\{\big((z,r),(z',r')\big)\in\V^2\colon r+z'=r'\;\mbox{and}\;\bar\D(r')-\bar\D(r)+\nu(z'-z)+\frac 1 2 \gamma\eta\geq 0\right\},
 \end{array}$$ 
 where we use the following notations: 
 $$\alpha=\inf_{t\in[0,T)}D'_+(t),\qquad R=\left\lfloor\frac{D(T)M}C\right\rfloor,\qquad\eta=\frac C M,\qquad\mbox{and}\qquad\gamma=2\left(\frac 1\alpha+2\nu\right).$$
Set for each arc $a=\big((z,r),(z',r')\big)$ a weight $w(a)=f^{\ave}\big(\bar\D(r')+\nu(z'-\eta),r+\eta,r'\big)$. \\

If $CS<D(T)$, there is no feasible solution. We can thus assume that $CS\geq D(T)$. The algorithm consists first in computing a path $\tilde p$ minimizing $\sum_{a\in\A(p)}w(a)$, among all paths $p$ with at most $S$ arcs starting at $(0,0)\in\V$ and ending at a vertex $(z,r)$ with $r=R\eta$. Such paths exist, see Lemma~\ref{lem:kopt} below. The computation of $\tilde p$ can be done in $O(S|\A|)$ via dynamic programming. Let the vertex sequence of $\tilde p$ be $\big((z_0,r_0),(z_1,r_1),\ldots,(z_n,r_n)\big)$. The algorithm consists then in defining recursively
$$\tilde y_j=\left\{\begin{array}{ll} 
0 & \mbox{for $j=0$} \\ 
\min\big(r_j+\eta,\tilde y_{j-1}+C,D(T)\big) & \mbox{for $j=1,\ldots,n$} \smallskip\\
D(T) & \mbox{for $j=n+1,\ldots,S$}
\end{array}\right.$$
and
$$\tilde d_j=\left\{\begin{array}{ll} 
\bar\D(\tilde y_j)+\nu(\tilde y_j-\tilde y_{j-1})+j\gamma\eta & \mbox{for $j=1,\ldots,n$}  \smallskip\\
\max(\tilde d_n,T+\nu(\tilde y_{n+1}-\tilde y_n)) &\mbox{for $j=n+1,\ldots,S$}
\end{array}\right.$$
and outputting the pair $(\tilde\dd,\tilde\yy)$. The construction of the graph is sketched on Figure~\ref{fig:algo}.\\

\begin{figure}
\begin{center}
\includegraphics[width=12cm]{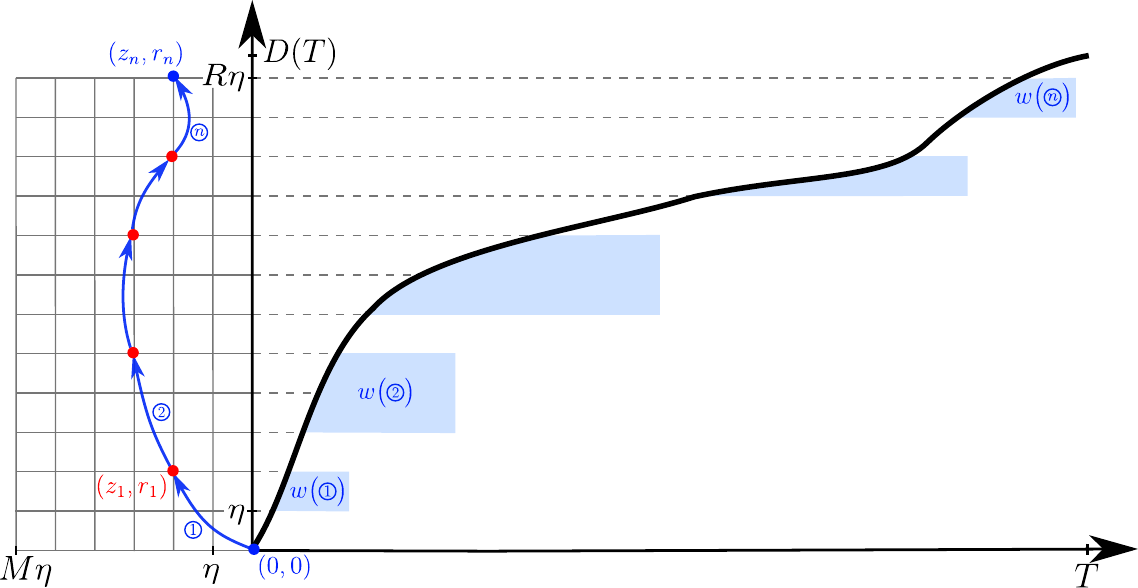}
\caption{\label{fig:algo} A feasible path in the algorithm proposed for solving \Pavenoret{}.}
\end{center}
\end{figure}

As it will be shown below, this $(\tilde\dd,\tilde\yy)$ is a feasible solution of \Pavenoret{} providing a value to the objective function within a $O\left(\frac {S^2} M\right)$ gap to the optimal value.

\subsubsection{Proof of Theorem~\ref{thm:pave}}

We provide three lemmas, which are proved in a separate section at the very end of the paper to ease the reading. Theorem~\ref{thm:pave} results immediately from their combination.

In the proofs of the lemmas, we assume that $M$ is large enough so that $\eta<D(T)$. Since in Theorem~\ref{thm:pave}, $M$ appears in `big O' formulas, it is a valid assumption. Anyway it is what is sought in practice: the larger $M$, the larger the accuracy of the solution. An $\eta$ of same order of magnitude of $D(T)$ would be useless.

\begin{lemma}\label{lem:kopt} For every optimal solution $(\dd^*,\yy^*)$, there is a path $p$ with at most $S$ arcs starting at $(0,0)\in\V$ and ending at a vertex $(z,r)$ with $r=R\eta$ and such that 
$$\frac 1 {D(T)}\sum_{a\in A(p)}w(a)\leq \gave(\dd^*,\yy^*).$$
\end{lemma}

\begin{lemma}\label{lem:feas}
The pair $(\tilde\dd,\tilde\yy)$ is a feasible solution of \Pavenoret{}.
\end{lemma}

\begin{lemma}\label{lem:byy} The following inequality holds:
$$\gave(\tilde\dd,\tilde\yy)\leq\frac 1 {D(T)}\sum_{a\in A(\tilde p)}w(a)+O\left(\frac {S^2} M\right).$$
\end{lemma}

\subsection{When the demand function is a step function}

 Better complexity results can be obtained when the demand is a step function. A {\em step function} is a function that can be written as a finite linear combination of indicator functions of intervals. The assumption on $D(\dot)$ being a step function means that the users arrive only on a finite number of instants. As it has already been noted, the assumption $\nu=0$ is equivalent to the assumption that every user boards a shuttle as soon as he arrives in the terminal.
 
 \begin{proposition}\label{prop:pmax}
Assume that $D(\dot)$ is a step function defined with $K$ discontinuities, supposed to be part of the input. Suppose moreover that $\nu=0$. Then for each of \Pmaxnoret{} and of  \Pavenoret{}, there is an algorithm computing an optimal solution in $O(K^2S)$.
\end{proposition}

It turns out that when $C$ and the values taken by $D(\dot)$ are integer, the loads of the shuttles in the optimal solution returned by the algorithm are also integer. We cover thus the case where the users are atoms.

\subsubsection*{The algorithm} We provide only the algorithm for \Pmaxnoret, the other case can be dealt similarly. Let $t_1<\cdots<t_K$ be the $K$ discontinuities. Define the directed graph $\G=(\V,\A)$ by 
$$\begin{array}{rcl}
\V & = & \{0\}\cup\left\{D(t_k)+Cq\colon k\in\left\{1,\ldots,K\right\}, q\in\{0,1,\ldots,Q\}\right\} \smallskip\\
\A & = & \{(y,y')\in\V^2\colon 0\leq y'-y\leq C\},
\end{array}$$
where $Q=\lfloor D(T)/C\rfloor$. Note that the vertex set is a finite subset of $\R_+$. Set for each arc $a=(y,y')$ a weight $w(a)=\bar\D(y')-\D(y)$. We consider the two vertices $0$ and $D(T)$ (obtained with $k=K$ and $q=0$). \\

If $CS<D(T)$, there is no feasible solution. We can thus assume that $CS\geq D(T)$. The algorithm consists first in computing a $0$-$D(T)$ path $\tilde p$ with $S$ arcs minimizing $\max_{a\in \A(\tilde p)}w(a)$. Within the proof of Proposition~\ref{prop:pmax} below, we show that from any feasible solution we can build a $0$-$D(T)$ path with $S$ arcs in $\G$. Thus, when the problem is feasible, such paths exist in $\G$.  The computation of $\tilde p$ can be done in $O(S|\A|)$ via dynamic programming. Let the vertex sequence of $\tilde p$ be $(\tilde y_0,\tilde y_1,\ldots,\tilde y_S)$. The end of the algorithm consists in defining $\tilde d_j=\bar\D(\tilde y_j)$ for all $j\in\{1,\ldots,S\}$ and outputting the pair $(\tilde\dd,\tilde\yy)$.

As it will be shown below, this $(\tilde\dd,\tilde\yy)$ is an optimal solution of \Pmaxnoret{}.
 
\subsubsection*{Proof of Proposition~\ref{prop:pmax}} According to Claim~\ref{claim:change_obj}, we replace the objective function of \Pmaxnoret{} by $\max_{j\in\left\{1,\ldots,S\right\}}(d_j-\D(y_{j-1}))$. It can easily be checked that $(\tilde\dd,\tilde\yy)$ is feasible for \Pmaxnoret. It provides a value $\max_{j\in\left\{1,\ldots,S\right\}}(\bar\D(\tilde y_j)-\D(\tilde y_{j-1}))$ for \Pmaxnoret{} (with the alternative objective function), and this value coincides with $\max_{a\in \A(\tilde p)}w(a)$. The path $\tilde p$ describes therefore a solution of \Pmaxnoret{} with a value equal to $\max_{a\in \A(\tilde p)}w(a)$.

Conversely, let $(\dd,\yy)$ be any feasible solution of \Pmaxnoret. Let $\bar\yy$ be the sequence defined by $\bar y_j=\min\{y\in\V\colon y\geq y_j\}$.  On the one hand, we have $\bar y_{j-1}\leq \bar y_j$ because $y_{j-1}\leq y_j$. On the other hand, we have $\bar y_{j-1}+C\geq y_{j-1}+C\geq y_j$.
If $\bar y_{j-1}+C\in\V$, we have $\bar y_{j-1}+C\geq\bar y_j$ by definition of $\bar y_j$. If $\bar y_{j-1}+C\notin\V$, then $\bar y_{j-1}+C> D(T)\geq\bar y_j$ since $D(T)\in\V$. Thus, $(\bar y_{j-1},\bar y_j)\in\A$ for all $j\in\{1,\ldots,S\}$. We have $\bar y_0=0$ and $\bar y_S=D(T)$ and the sequence $\bar \yy$ is a $0$-$D(T)$ path $p$ with $S$ arcs.

Second, we prove that $\bar\D(\bar y_j)-\D(\bar y_{j-1})\leq d_j-\D( y_{j-1})$ as follows. There exists a unique $k$ such that $D(t_k)<y_j\leq D(t_{k+1})$. By definition of $D(\dot)$, we have $D(t)=D(t_k)$ for all $t\in[t_k,t_{k+1})$, and thus $\bar\D(y_j)=t_{k+1}$. Since $D(t_{k+1})\in\V$, we have $\bar y_j\leq D(t_{k+1})$ by definition of $\bar y_j$, and hence $\bar\D(\bar y_j)\leq t_{k+1}$ (directly by definition of $\bar\D(\dot)$). Therefore, $\bar\D(\bar y_j)-\D(\bar y_{j-1})\leq\bar\D(y_j)-\D(y_{j-1})\leq d_j-\D(y_{j-1})$ (where we use the fact that $\D(\dot)$ is nondecreasing). 

Finally, we have $\max_{a\in\A(p)}w(a)\leq \max_{j\in\{1,\ldots,S\}}\big(d_j-\D( y_{j-1})\big)$. As the path $\tilde p$ is optimal, $\max_{a\in\A(\tilde p)}w(a)$ is a lower bound on the value taken by the (alternative) objective function on $(\dd,\yy)$. \qed

\section{When return is allowed}\label{sec:ret}

\subsubsection*{The algorithm} The following map will be useful:
$$f^{\max}\colon(\ell,y,y')\longmapsto\left\{\begin{array}{ll} \max(\ell,\bar\D(y'))+\nu(y'-y)-\D(y) & \mbox{if $y'\geq y$} \\ 0 & \mbox{if $y'=y$.}\end{array}\right.$$

We introduce the following two sets
$$\begin{array}{rcl}
\setQ & = & \{0,\eta,\ldots,(\left\lfloor T^+/\eta\right\rfloor+1)\eta\}^S \smallskip\\
 \setR & = & \left\{\rr\in\{0,\eta,\ldots,R\eta\}^S\colon 0\leq r_k-r_{k-1}\leq M\eta\mbox{ for $k=2,\ldots,S$}\right\},
 \end{array}$$ where 
$$\eta=\frac{C}{M},\qquad R=\left\lfloor\frac{D(T)M}C\right\rfloor,\qquad\mbox{and}\qquad T^+=T+\frac{\nu D(T)}S+\left(\left\lceil\frac{D(T)}{CS}\right\rceil-1\right)\pi.$$
Define the directed graph $\G=(\V,\A)$ by 
$$\begin{array}{rcl} 
 \V & = &\left\{(z,\qq,\rr)\in\{0,\eta,\ldots,M\eta\}\times\setQ\times\setR\colon r_k\leq D(q_k)\;\mbox{for $k=1,\ldots,S$}\right\}\medskip\\
 \A & = & \left\{\big((z,\qq,\rr),(z',\qq',\rr')\big)\in \V^2\;\mbox{satisfying $(\star)$}\right\},
 \end{array}
 $$
 where $$(\star)\quad r_S+z'=r_1'\quad\mbox{and}\quad q_k'-q_k-\nu(r_k-r_{k-1}) -\pi+(1+\nu)\eta\geq 0\mbox{ for $k=1,\ldots, S$}.$$
We adopt the convention $D(t)=D(T)$ when $t\geq T$ and we define $r_0=r_1-z$. Set for each arc $a=\big((z,\qq,\rr),(z',\qq',\rr')\big)$ a weight  
$w(a)=\max_{k\in\left\{1,\ldots,S\right\}}f^{\max}(q_k'-\eta,r_{k-1}'+\eta,r'_k)$
 where $r'_0=r'_1-z'$.\\

The algorithm consists first in computing a path $\tilde p$ minimizing $\max_{a\in A(p)}w(a)$ among all paths $p$ starting at $(0,\zero,\zero)\in\V$ (the `all zero' vector) and ending at a vertex $(z,\qq,\rr)$ with $r_S=R\eta$. Such paths exist, see Lemma~\ref{lem:loptmax} below. It can be done in $O(|\V||\A|)$ via dynamic programming. Let the vertex sequence of $\tilde p$ be $\big((0,\zero,\zero),(z_0,\boldsymbol{q}^{0},\boldsymbol{r}^{0}),\ldots,(z_n,\boldsymbol{q}^n,\boldsymbol{r}^{n})\big)$. The vector $\boldsymbol{r}^i$ models the cumulative loads of the $S$ shuttles when they perform their $i$th departure. The vector $\boldsymbol{q}^{i}$ models the times at which the loading of the $S$ shuttles starts when they perform their $i$th departure. These quantities are computed only approximatively (with an accuracy $\eta$).

The algorithm consists then in defining recursively
for all $j=iS+k$ with $i=0,\ldots,n$ and $k=1,\ldots,S$
$$\begin{array}{rcl}
\tilde y_{j}&=&\left\{\begin{array}{ll}\min\big(r_{k}^i+\eta,y_{j-1}+C,D(T)\big)&\mbox{if } r^i_k>r^i_{k-1}\\\tilde y_{j-1}&\mbox{otherwise}\end{array}\right.\\
\tilde d_{j}&=&\max(q_{k}^i,\bar\D(\tilde y_{j}))+j\tilde\gamma\eta+\nu(\tilde y_{j}-\tilde y_{j-1})
\end{array}$$
where $\tilde y_0=0$, $r^i_0=r^{i-1}_S$, $r^0_0=0$, and $\tilde\gamma=(1+2\nu+1/\alpha)$. For $j=(n+1)S+1,\ldots,N$
$$\begin{array}{rcl}
\tilde y_{j}&=&D(T)\\
\tilde d_{j}&=&\max\big(\tilde d_{j-S}+\pi,T\big)+\nu(\tilde y_j-\tilde y_{j-1}).
\end{array}$$
and outputting the pair $(\tilde\dd,\tilde\yy)$.\\

As it will be stated below, this $(\tilde\dd,\tilde\yy)$ is a feasible solution of \Pmaxret{} providing a value to the objective function $\gmax(\dot)$ within a $O\left(\frac {S^2} M\right)$ gap to the optimal value.

\subsubsection*{Proof of Theorem~\ref{thm:pmaxret}}

We provide four lemmas. The proof of Lemma~\ref{lem:t+} is almost the one of Proposition~\ref{prop:finite_max} and the proofs of the three others follow the same scheme as the ones of Lemmas~~\ref{lem:kopt},~\ref{lem:feas}, and~\ref{lem:byy}. They are thus omitted. Theorem~\ref{thm:pmaxret} results immediately from their combination.

\begin{lemma}\label{lem:loptmax}
For every optimal solution $(\dd^*,\yy^*)$, there is a path $p$ starting at $(0,\zero,\zero)\in\V$ and ending at a vertex $(z,\qq,\rr)$ with $r_S=R\eta$ and such that 
$$\max_{a\in A(p)} w(a)\leq \gmax(\dd^*,\yy^*).$$
\end{lemma}

\begin{lemma}\label{lem:t+}
There is an optimal solution of \Pmaxret{} for which $T^+$ is an upper bound on the loading time of the last departure.
\end{lemma}


\begin{lemma}\label{lem:feasret}
The pair $(\tilde \dd, \tilde \yy)$ is a feasible solution of \Pmaxret.
\end{lemma}

\begin{lemma}\label{lem:bllyymax}
The following inequality holds:
\begin{eqnarray*}
\gmax(\tilde\dd,\tilde\yy)&\leq&\max_{a\in A(\tilde p)} w(a)+O\left(\frac {S^2}M\right)\\
\end{eqnarray*}

\end{lemma}

\section{Experimental results}\label{sec:experiments}

In this section, we test the performance of the algorithms described in previous sections for problems \Pmaxnoret{}, \Pavenoret{}, and \Pmaxret. As explained in Section~\ref{sec:mainresults}, we do not have such an algorithm for problem \Paveret. 

\subsection{Data}

Our experiments are based on real data provided by our partner Eurotunnel. They are related to the transportation of freight trucks between France and Great Britain. Some parameters are fixed as constants of the problems and do not vary from an instance to another. For the constants $C, T, \nu ,\pi$ of our problem, we take the real values used in practice by the company:


 

 $$
C=32,\quad T=1440\min \mbox{ (one day)}, \quad \nu=0.625\min,\quad \pi=34\min.
 $$ 
 (The value taken for $\pi$ is actually the duration of a trip going from France to Great-Britain, and not of the round trip, which lasts approximatively twice this quantity.)
 
Two functions $D(\dot)$ are used. The first one (``1P'') is a piecewise affine map obtained by averaging the real demand over several days. It turns out that this function has a peak period in the morning. The second function (``2P''), also piecewise affine, is obtained from the first by artificially adding a second peak period in the evening. In both cases, $D(\dot)$ is increasing and $D(T)=2016$. For problems \Pmaxnoret{} and \Pavenoret{}, we consider $S\in[100,250]$ since the number of shuttle trips in every direction is within this range for a typical day.
 
 The numerical experiments are performed on a Macbook Pro of 2014 with four 2.2 Ghz processors and 16 Gb of ram.
 
\subsection{Results}

The problems \Pmaxnoret{}, \Pavenoret{}, and \Pmaxret{} are solved with algorithms described in this article. The results are summarized in the following tables. 

Table~\ref{tab:Pmaxnoret} gives the numerical results for problem \Pmaxnoret. The next column is the number of shuttles $S$ in the fleet. The third column provides the parameter $\eps$ of the algorithm, which is an \textit{a priori} upper bound on the optimality gap (Corollary~\ref{cor:approx}). The two next columns give respectively the lower bound and the upper bound (value of the feasible solution returned by the algorithm), both expressed in minutes. The next column is the optimality gap. The last column provides the CPU time spent solving the problem.

\begin{table}[h]
\begin{tabular}{rr|r|rrr|r}
\multicolumn{1}{c}{$D$} & \multicolumn{1}{c|}{$S$} & \multicolumn{1}{c|}{$\eps$} & \multicolumn{1}{c}{LB} & \multicolumn{1}{c}{UB} & \multicolumn{1}{c|}{gap} & \multicolumn{1}{c}{CPU}\\&&&&&\multicolumn{1}{c|}{(\%)}&\multicolumn{1}{c}{(\second)}\\
\hline
1P & 100 &  $10^{-4}$  &27.2 & 27.2 & 0.0 & 0 \\
1P & 150 &  $10^{-4}$  &18.1 & 18.1 & 0.0 & 0 \\
1P & 200 &  $10^{-4}$  &13.6 & 13.6 & 0.0 & 0 \\
1P & 250 &  $10^{-4}$  &11.0 & 11.0 & 0.0 & 0 \\
\hline
2P & 100 &  $10^{-4}$  &27.0 & 27.0 & 0.0 & 0 \\
2P & 150 &  $10^{-4}$  &18.0 & 18.0 & 0.0 & 0 \\
2P & 200 &  $10^{-4}$  &13.5 & 13.5 & 0.0 & 0 \\
2P & 250 &  $10^{-4}$  &10.8 & 10.8 & 0.0 & 0 \\
\end{tabular}
\bigskip
\caption{Numerical results for problem \Pmaxnoret\label{tab:Pmaxnoret}}
\end{table}

Table~\ref{tab:Pavenoret} gives the numerical results for problem \Pavenoret. The columns are the same as for Table~\ref{tab:Pmaxnoret} except the third one which provides here the parameter $M$ of the algorithm. We know from Theorem~\ref{thm:pave} that the gap between the upper bound and the lower bound converges asymptotically to 0 when $M$ goes to infinity. We tried $M=32$ and $M=128$.

\begin{table}[h]
\begin{tabular}{rr|r|rrr|r}
\multicolumn{1}{c}{$D$} & \multicolumn{1}{c|}{$S$} & \multicolumn{1}{c|}{$M$} & \multicolumn{1}{c}{LB} & \multicolumn{1}{c}{UB} & \multicolumn{1}{c|}{gap} & \multicolumn{1}{c}{CPU}\\&&&&&\multicolumn{1}{c|}{(\%)}&\multicolumn{1}{c}{(\second)}\\
\hline
1P & 100 &  32  &17.3 & 19.2 & 10.0 & 34 \\ 
1P & 100 &  128  &18.7 & 19.2 & 2.5 & 1930 \\
\hline
1P & 200 &  32  &7.7 & 9.6 & 19.4 & 70 \\
1P & 200 &  128  &9.1 & 9.6 & 5.0 & 4035 \\
\hline
2P & 100 &  32  &17.5 & 19.4 & 9.9 & 38 \\
2P & 100 &  128  &18.9 & 19.4 & 2.5 & 2387	 \\
\hline
2P & 200 &  32  &7.9 & 9.7 & 19.2 & 76 \\
2P & 200 &  128  &9.2 & 9.7 & 5.0 & 4463 \\
\end{tabular}
\bigskip
\caption{Numerical results for problem \Pavenoret\label{tab:Pavenoret}}
\end{table}

Table~\ref{tab:Pmaxret} gives the numerical results for problem \Pmaxret. The columns are the same as for Table~\ref{tab:Pavenoret}. Since the computation time was prohibitively long as soon as $S\geq 2$, we made experiments for $S=1$. To get realistic waiting times for the users, we divided the demand functions by 3.5 leading to (``1P$^*$'') and (``2P$^*$''). Again, we know from Theorem~\ref{thm:pmaxret} that for large $M$, we will be close to the optimal solution and we tried $M=16$ and $M=32$.

\begin{table}[h]
\begin{tabular}{rr|r|rrr|r}
\multicolumn{1}{c}{$D$} & \multicolumn{1}{c|}{$S$} & \multicolumn{1}{c|}{$M$} & \multicolumn{1}{c}{LB} & \multicolumn{1}{c}{UB} & \multicolumn{1}{c|}{gap} & \multicolumn{1}{c}{CPU}\\&&&&&\multicolumn{1}{c|}{(\%)}&\multicolumn{1}{c}{(\second)}\\
\hline
1P$^*$ & 1 &  16  &168.6 & 214.2 & 21.3 & 104 \\
1P$^*$ & 1 &  32  &184.5 & 210.8 & 12.5 & 1654 \\
\hline
2P$^*$ & 1 &  16  &101.0 & 131.0 & 22.9 & 106\\
2P$^*$ & 1 &  32  &106.9 & 126.3 & 15.4 & 1848 \\
\end{tabular}
\bigskip
\caption{Numerical results for problem \Pmaxret\label{tab:Pmaxret}}
\end{table}

\subsection{Comments}

In Table~\ref{tab:Pmaxnoret}, the results for problem \Pmaxnoret{} are extremely convincing, the optimal solutions were found almost immediately.
In Table~\ref{tab:Pavenoret}, the algorithm for problem \Pavenoret{} was able to find provable good solutions within reasonable computation times. We may note that increasing $M$ after some threshold does not seem to improve the quality of the return solution. This was confirmed by other experiments not shown here. It may indicate that the algorithm could be used efficiently in practice.
In Table~\ref{tab:Pmaxret}, the same holds for \Pmaxret{} once we have accepted to work with one shuttle. Finding an efficient algorithm with at least two shuttles seems to remain a challenging task.

\section{Proofs of Lemmas of Section~\ref{subsec:pave}}

\begin{claim}\label{lem:low_y}
We have $r_j\leq \tilde y_j\leq r_j+\eta$ for $j=0,\ldots,n$. 
\end{claim}

\begin{proof}
We have $\tilde y_j\leq r_j+\eta$ by definition. Using $r_j-r_{j-1}\leq M\eta$ in a feasible path, a direct induction shows that $\tilde y_j\geq r_j$ for $j=0,\ldots,n$. 
\end{proof}

\begin{claim}\label{lem:infbarD}
Suppose that $\alpha>0$. Then for all $y\in[0,D(T)]$ and $\delta\in[0,D(T)-y]$, we have $\bar\D(y+\delta)\leq\bar\D(y)+\delta/\alpha$ and $\D(y+\delta)\leq\D(y)+\delta/\alpha$.
\end{claim}

\begin{proof}
Diewert~\cite{diewert1981alternative} extended the Mean Value Theorem to semicontinuous functions. According to his result, 
for any $0\leq a\leq b\leq T$, there exists $c\in[a,b)$ such that
$$
\limsup_{t\to 0^+} \frac{D(c+t)-D(c)}{t}\leq \frac{D(b)-D(a)}{b-a}.
$$
Since $$\alpha=\inf_{t\in[0,T)}D'_+(t)\leq D'_+(c)=\limsup_{t\to 0^+} \frac{D(c+t)-D(c)}{t},$$ we have $D(a)+\alpha(b-a)\leq D(b)$ for any $0\leq a\leq b\leq T$.  With $a=\bar\D(y)$ and $b=\bar\D(y)+\delta/\alpha$, we get $y+\delta\leq D(\bar\D(y))+\delta\leq D(\bar\D(y)+\delta/\alpha)$ (the first inequality is given by Lemma~\ref{lem:pseudo}). By definition of $\bar\D$, we have $\bar\D(y+\delta)\leq\bar\D(y)+\delta/\alpha$.
The second inequality is proved along the same lines.
\end{proof}

\begin{proof}[Proof of Lemma~\ref{lem:kopt}]
Let $(\dd^*,\yy^*)$ be an optimal solution of \Pavenoret{} such that $d_j^*=\bar\D(y_j^*)+\nu(y_j^*-y_{j-1}^*)$ for all $j\in\{1,\ldots,S\}$ (Claim~\ref{claim:remove_dj}). Consider the sequence $\lfloor y^*_1/\eta\rfloor\eta,\ldots,\lfloor y^*_S/\eta\rfloor\eta$ and remove the repetitions. Since the sequence is nondecreasing, we obtain an increasing sequence $\boldsymbol{r}=r_1,\ldots,r_n$. We introduce $\sigma\colon\{1,\ldots,n\}\rightarrow\{1,\ldots,S\}$ with $\sigma(j)$ being the smallest index such that $r_j=\lfloor y_{\sigma(j)}^*/\eta\rfloor\eta$. We then define $z_j=r_j-r_{j-1}$ for $j\in\left\{1,\ldots,n\right\}$, with $r_0=0$. We prove that the sequence $(z_j,r_j)_{j\in\left\{1,\ldots,n\right\}}$ provides a feasible path from the vertex $(0,0)$ to $(z_n,r_n)$ in $\G$.  First note that $r_n=R\eta$ since $y_S^*=D(T)$ and that $z_j>0$.  For all $j\in\left\{1,\ldots,n\right\}$, we have $z_j=r_j-r_{j-1}=\left(\lfloor y_{\sigma(j)}^*/\eta\rfloor-\lfloor y_{\sigma(j)-1}^*/\eta\rfloor+\lfloor y_{\sigma(j)-1}^*/\eta\rfloor-\lfloor y_{\sigma(j-1)}^*/\eta\rfloor\right)\eta< M\eta+\eta$, since $\lfloor y_{\sigma(j)-1}^*/\eta\rfloor=\lfloor y_{\sigma(j-1)}^*/\eta\rfloor$ and $y^*_{\sigma(j)}-y^*_{\sigma(j)-1}\leq C$. Thus $z_j\leq M\eta$. Moreover by definition, $r_j\leq R\eta$. Therefore $(z_j,r_j)\in\V$ for all $j\in\left\{1,\ldots,n\right\}$. Let us now prove that $((z_{j-1},r_{j-1}),(z_j,r_j))\in\A$ for all $j\in\left\{2,\ldots,n\right\}$. By definition, $z_j+r_{j-1}=r_j$. 
Note that because of the definition of $r_j$, we have $r_j\leq y_{\sigma(j)}^*\leq y_{\sigma_{(j+1)}-1}^*<r_j+\eta$. 
Combining these inequalities for all $j$ with Claim~\ref{lem:infbarD} leads to
\begin{eqnarray*}
\bar\D(r_j)-\bar\D(r_{j-1})+\nu(z_j-z_{j-1})&\geq & \bar\D(y_{\sigma(j)}^*)-\eta/\alpha-\bar\D(y_{\sigma(j-1)}^*) \\
& & \qquad+\nu(y_{\sigma(j)}^*-y_{\sigma(j)-1}^*-y_{\sigma(j-1)}^*+y_{\sigma(j-1)-1}^*-2\eta)\\
& = & d_{\sigma(j)}^*-d_{\sigma(j-1)}^*-(1/\alpha+2\nu)\eta \\
&\geq& -(1/\alpha+2\nu)\eta.
\end{eqnarray*}
The sequence $(z_j,r_j)_{j\in\left\{1,\ldots,n\right\}}$ is then a feasible path $p$ from the vertex $(0,0)$ to $(z_n,r_n)$ in $\G$, with at most $S$ arcs. The only thing that remains to be checked in that the claimed inequality holds.

We have $f^{\ave}( d_{\sigma(j)}^*, y_{\sigma(j)-1}^*, y_{\sigma(j)}^*)\geq f\big(\bar\D(r_j)+\nu(z_j-\eta),r_{j-1}+\eta, r_j\big)$ for all $j\in\left\{1,\ldots,n\right\}$ since $f^{\ave}(\dot)$ is nonincreasing in the second term and nondecreasing in the first and third terms. Thus, $$\frac 1 {D(T)}\sum_{a\in A(p)}w(a)\leq\frac 1 {D(T)}\sum_{j=1}^nf^{\ave}( d_{\sigma(j)}^*, y_{\sigma(j)-1}^*, y_{\sigma(j)}^*)
\leq\gave(\dd^*,\yy^*).$$ Since this inequality holds for any optimal solution of \Pavenoret{}, we get the conclusion.
\end{proof}

\begin{proof}[Proof of Lemma~\ref{lem:feas}]
We are going to check that $(\tilde\dd,\tilde\yy)$ is feasible for \Pavenoret{}. 

For $j=1,\ldots,n$, we have $\tilde y_j-\tilde y_{j-1}\leq C$ by definition of $\tilde \yy$. For $j=n+2,\ldots,S$, we have $\tilde y_j-\tilde y_{j-1}=0$. Finally, we have $\tilde y_{n+1}-\tilde y_n\leq D(T)-r_n<\eta\leq C$ (where we use Claim~\ref{lem:low_y} to bound $\tilde y_n$). Thus, $\tilde \yy$ satisfies constraint (i).

For $j=1,\ldots,n$, if $r_j>r_{j-1}$, we have $\tilde y_{j-1}\leq r_{j-1}+\eta\leq r_j\leq \tilde y_j$ (the last inequality being Claim~\ref{lem:low_y}) and if $r_j=r_{j-1}$, necessarily $r_j=r_{j-1}=0$ and $\tilde y_{j-1}=\tilde y_j=\eta$.
Thus, $\tilde \yy$ satisfies constraint (ii).

Consider $j\in\{2,\ldots,n\}$. We have 
\begin{eqnarray*}
\tilde d_j-\tilde d_{j-1} & = & \bar\D(\tilde y_j)+\nu(\tilde y_j-\tilde y_{j-1})-\D(\tilde y_{j-1})-\nu(\tilde y_{j-1}-\tilde y_{j-2})+\gamma\eta  \\
& \geq & \bar\D(r_j)-\bar\D(r_{j-1}+\eta)+\nu(r_j-2r_{j-1}+r_{j-2}-2\eta)+\gamma\eta \\
& \geq & \bar\D(r_j)-\bar\D(r_{j-1})-\eta/\alpha+\nu(z_j-z_{j-1}-2\eta)+\gamma\eta\\
& \geq & 0.
\end{eqnarray*} The first inequality is obtained with the help of Claim~\ref{lem:low_y}. For the second one, we use Claim~\ref{lem:infbarD} and also that $z_j=r_j-r_{j-1}$ and $z_{j-1}=r_{j-1}-r_{j-2}$ which hold because $\tilde p=\big((z_0,r_0),(z_1,r_1),\ldots,(z_n,r_n)\big)$ is a path in $\G$. For the last inequality, we use $\bar\D(r_j)-\bar\D(r_{j-1})+\nu(z_j-z_{j-1})+\frac 1 2 \gamma\eta\geq 0$, which holds again because $\tilde p$ is a path, and the definition of $\gamma$. For $j\geq n+1$, we have $\tilde d_j\geq \tilde d_{j-1}$ by definition. Constraint (iii) is thus satisfied for all $j$.

If $n<S$, then $\tilde y_S=D(T)$ by definition. From now on, we suppose thus that $n=S$. We also suppose that $S\geq 2$. The case $S=1$ being easy to check (and anyway, for a complexity point of view, this case does not matter). If $\tilde y_{S-1}=r_{S-1}+\eta$, then $\tilde y_{S-1}+C=r_{S-1}+\eta+C\geq r_S+\eta>D(T)$ (here we use that $z_S\leq C$ and that $r_S=R\eta$) and thus $\tilde y_S=D(T)$. If $\tilde y_{S-1}=D(T)$, then $\tilde y_S=D(T)$ since $\tilde y_{S-1}\leq \tilde y_S\leq D(T)$. Hence, in all these cases, $\tilde \yy$ satisfies constraint (iv). The only remaining case is when $\tilde y_{S-1}=\tilde y_{S-2}+C$. If $j$ is an index in $\left\{1,\ldots,S-2\right\}$ such that $\tilde y_j=r_j+\eta$, then we have $r_{j+1}+\eta\leq r_j+C+\eta=\tilde y_j+C$ and $r_{j+1}+\eta\leq D(T)$, and thus $\tilde y_{j+1}=r_{j+1}+\eta$. It implies that as soon as some $j_0\in\left\{1,\ldots,S-1\right\}$ is such that $\tilde y_{j_0}=r_{j_0}+\eta$, we have $\tilde y_{S-1}=r_{S-1}+\eta$, which is a case we have already dealt with. Since $r_j+\eta\leq r_S\leq D(T)$ for $j\in\{1,\ldots,S-1\}$, we are left with the case where $\tilde y_j=\tilde y_{j-1}+C$ for every $j\in\left\{1,\ldots,S-1\right\}$. In this situation, we have $\tilde y_{S-1}=(S-1)C$ 
and hence $\tilde y_{S-1}+C=CS\geq D(T)$. Since $r_S+\eta> D(T)$, we get that $\tilde y_S=D(T)$, and $\tilde \yy$ satisfies constraint (iv) in every case.

For $j=1,\ldots,n$, we have $\tilde d_j\geq \bar\D(\tilde y_j)+\nu(\tilde y_j-\tilde y_{j-1})$ by definition, and for $j\geq n+1$, we have $\tilde d_j\geq T+\nu(\tilde y_{n+1}-\tilde y_n)\geq\bar\D(\tilde y_j)+\nu(\tilde y_j-\tilde y_{j-1})$. Thus $\tilde \dd$ satisfies constraint (v) and $(\tilde \dd,\tilde \yy)$ is feasible for \Pavenoret{}.
\end{proof}

\begin{proof}[Proof of Lemma~\ref{lem:byy}]
Our goal is to bound from above the following quantity \begin{equation}\label{eq:g}
\gave(\tilde \dd,\tilde \yy)=\frac 1 {D(T)}\sum_{j=1}^Sf^{\ave}(\tilde d_j,\tilde y_{j-1},\tilde y_j)
\end{equation} 
We proceed by splitting the expression into two parts: the sum from $j=1$ to $j=n$, and the sum from $j=n+1$ to $j=S$.

Using Claims~\ref{lem:low_y} and~\ref{lem:infbarD}, we have 
$\bar\D(\tilde y_j)+\nu(\tilde y_j-\tilde y_{j-1})\leq q_j+\eta/\alpha+\nu\eta$, where $q_j=\bar\D(r_j)+\nu(r_j-r_{j-1})$.
Thus we have for all $j\leq n$,
\begin{equation}\label{eq:B1}
\sum_{j=1}^nf^{\ave}(\tilde d_j,\tilde y_{j-1},\tilde y_j)\leq \sum_{j=1}^nf^{\ave}(q_j+\eta/\alpha+\nu\eta+j\gamma\eta,r_{j-1},r_j+\eta),
\end{equation}
since $f^{\ave}(\dot)$ is nonincreasing in the second term and nondecreasing in the first and third terms and where we extend the definition of $\bar\D(\dot)$ by letting $\bar\D(y)=T$ for all $y>D(T)$. 

For the second part, we proceed as follows. Since $r_n+\eta=(R+1)\eta>D(T)$, Claim~\ref{lem:low_y} immediately  implies $D(T)-\tilde y_n\leq\eta$. With Claim~\ref{lem:infbarD}, we get thus $T\leq\bar\D(\tilde y_n)+\eta/\alpha$, where we used $T=\bar\D\big(D(T)-\tilde y_n+\tilde y_n\big)$. This provides 
$$\tilde d_{n+1}\leq\bar\D(\tilde y_n)+\eta/\alpha +\nu(r_n-r_{n-1})+\nu\eta+n\gamma\eta=q_n+(1/\alpha +\nu+n\gamma)\eta.$$
Using again the fact that $f^{\ave}(\dot)$ is nonincreasing in the second term and nondecreasing in the first and third terms and with the help of Claim~\ref{lem:low_y}, we get
\begin{equation}\label{eq:B2}
\sum_{j=n+1}^Sf^{\ave}(\tilde d_{j},\tilde y_{j-1},\tilde y_{j})\leq f^{\ave}(q_n+\eta/\alpha+\nu\eta+n\gamma\eta,r_n,r_n+\eta),\end{equation}
since the terms indexed by $j=n+2,\ldots,S$ are all zero and since $D(T)<r_n+\eta$.

We aim at comparing the upper bounds in Equations~\eqref{eq:B1} and~\eqref{eq:B2} with 
\begin{equation}\label{eq:w}
\sum_{a\in A(\tilde p)}w(a)=\sum_{j=1}^nf^{\ave}(q_j-\nu\eta,r_{j-1}+\eta,r_j).
\end{equation}

We first compare the $j$th term of the bound in~\eqref{eq:B1} with the $j$th term of the sum in~\eqref{eq:w}.
$$
f^{\ave}(q_j+\eta/\alpha+\nu\eta+j\gamma\eta,r_{j-1},r_j+\eta)-f^{\ave}(q_j-\nu\eta,r_{j-1}+\eta,r_j)=I_j^1+I_j^2+I_j^3
$$ with
\begin{eqnarray*}
I_j^1 &=& \int_{r_{j-1}}^{r_{j-1}+\eta}\big(q_j+\eta/\alpha+\nu\eta+j\gamma\eta-\bar\D(u)\big)du\\
I_j^2&=&\int_{r_{j-1}+\eta}^{r_{j}}\big(j\gamma\eta+\eta/\alpha+2\nu\eta\big)du\\
I_j^3 &=& \int_{r_{j}}^{r_{j}+\eta}\big(q_j+\eta/\alpha+\nu\eta+j\gamma\eta-\bar\D(u)\big)du.
\end{eqnarray*}
Since $\bar\D(\dot)$ in nondecreasing, we get
\begin{eqnarray*}
I_j^1 &\leq& \big(\bar\D(r_j)-\bar\D(r_{j-1})+\nu(r_j-r_{j-1})\big)\eta+(1/\alpha+\nu+j\gamma)\eta^2\\
I_j^2&\leq& (r_j-r_{j-1})(j\gamma+1/\alpha+2\nu)\eta-(j\gamma+1/\alpha+2\nu)\eta^2\\
I_j^3&\leq& \nu(r_j-r_{j-1})\eta+(1/\alpha+\nu\eta+j\gamma)\eta^2.
\end{eqnarray*}
Using $j\gamma\leq n\gamma$ and $\gamma=2(1/\alpha+2\nu)$, we obtain
$$I_j^1+I_j^2+I_j^3\leq \big(\bar\D(r_j)-\bar\D(r_{j-1})+2\nu(r_j-r_{j-1})\big)\eta+(n+1/2)\gamma\eta^2+(r_j-r_{j-1})(n+1/2)\gamma\eta.$$
We now bound the term in Equation~\eqref{eq:B2}. Let $I=f^{\ave}(q_n+\eta/\alpha+\nu\eta+n\gamma\eta,r_n,r_n+\eta)$. We have
\begin{eqnarray*}
I &=& \int_{r_n}^{r_n+\eta}(q_n+\eta/\alpha+\nu\eta+n\gamma\eta-\bar\D(u))du.\\
&\leq & \nu(r_n-r_{n-1})\eta+(1/\alpha+\nu+n\gamma\big)\eta^2 .
\end{eqnarray*}

We have thus 
\begin{eqnarray*}
\gave(\tilde \dd,\tilde \yy)-\frac 1 {D(T)}\sum_{a\in A(\tilde p)}w(a) & \leq &  \frac 1 {D(T)}\left(\sum_{j=1}^n(I_j^1+I_j^2+I_j^3)+I\right) \\
& \leq & \frac 1 {D(T)}\left(\bar\D(r_n)+2\nu r_n+r_n(n+1)\gamma+\nu C+(n+1)^2\gamma \eta\right)\eta.
\end{eqnarray*}
Using $r_n\leq D(T)$ and $\bar\D(r_n)\leq T$ leads to
$$
\gave(\tilde \dd,\tilde \yy)\leq \frac 1{D(T)}\sum_{a\in A(p)}w(a)+\left(\frac{T+\nu C}{D(T)}+\gamma (S+1)\right)\eta+\frac {\gamma (S+1)^2}{D(T)}\eta^2.
$$
\end{proof}

\bibliographystyle{plain}
 \bibliography{database}

\end{document}